\documentclass[a4paper,reqno,11pt]{amsart}
\usepackage[T1]{fontenc}
\usepackage[utf8]{inputenc}
\usepackage{amssymb, amsmath, amsthm, graphicx, enumerate, enumitem, color}
\usepackage[usenames,dvipsnames,svgnames,table]{xcolor}
\usepackage[foot]{amsaddr}
\usepackage[normalem]{ulem}
\usepackage{thmtools}

\usepackage{caption}
\usepackage{subcaption}
\usepackage{wrapfig}
\input{insbox}

\makeatletter
\def\thm@space@setup{
\thm@preskip=4mm
\thm@postskip=0mm
}
\makeatother

\usepackage{tikz}
\usetikzlibrary{backgrounds}

\usepackage{hyperref}
\usepackage[capitalise, compress, noabbrev, nameinlink]{cleveref}
\definecolor{linkblue}{named}{MidnightBlue}
\hypersetup{colorlinks=true, linkcolor=linkblue,  anchorcolor=linkblue,
	citecolor=linkblue, filecolor=linkblue, menucolor=linkblue,
	urlcolor=linkblue} 
\usepackage{mathtools}
\usepackage{thmtools, thm-restate}
\usepackage[longnamesfirst,numbers,sort&compress]{natbib}
\usepackage{etoolbox}

\theoremstyle{plain}
\newtheorem{thm}{Theorem}
\newtheorem*{thm*}{Theorem}
\newtheorem{theorem}[thm]{Theorem}
\newtheorem{lem}[thm]{Lemma}
\newtheorem{lemma}[thm]{Lemma}
\newtheorem*{lemma*}{Lemma}
\newtheorem{cor}[thm]{Corollary}
\newtheorem*{cor*}{Corollary}

\newtheorem{claim}{Claim}[thm]

\newtheorem*{lem*}{Lemma}

\newtheorem*{conjecture*}{Conjecture}

\newenvironment{proofclaim}[1][]
    {\let\oldqed\qedsymbol\renewcommand{\qedsymbol}{\ensuremath{\lozenge}}\begin{proof}[Proof of claim] }{\end{proof}\renewcommand{\qedsymbol}{\oldqed}}

\newcommand{\R}{\mathbb{R}}
\let\le\leqslant

\let\leq\leqslant
\let\geq\geqslant

\let\subset\subseteq

\let\epsilon\varepsilon

\DeclareMathOperator\interior{int}
\DeclareMathOperator\Interior{\overline{int}}

\definecolor{brightmaroon}{rgb}{0.76, 0.13, 0.28}
\newcommand{\defin}[1]{\emph{\textcolor{brightmaroon}{#1}}}

\DeclarePairedDelimiter\floor{\lfloor}{\rfloor}

\newcommand\card[1]{\left|#1\right|}

\newcommand{\cgrid}[2]{$#1\times #2$ cylindrical grid}

\usepackage{comment}

\newcommand{\say}[1]{``#1''} 

\setenumerate{label=\textup{(\roman*)}, leftmargin=*, widest=iii,}
\setitemize{leftmargin=*, widest=iii, itemsep=.5em}

\title{An improved bound on the treewidth of planar graphs excluding a grid minor}

\begin{document}

\author[Cames van Batenburg]{Wouter Cames van Batenburg}
\address[W.~Cames van Batenburg]{D\'epartement d'Informatique, Universit\'e libre de Bruxelles, Belgium}
\email{w.p.s.camesvanbatenburg@gmail.com}

\author[Claus]{Quentin Claus}
\address[Q.~Claus]{D\'epartement de Mathématiques, Universit\'e libre de Bruxelles, Belgium}
\email{quentin.claus@ulb.be}

\author[Joret]{Gwena\"el Joret}
\address[G.~Joret]{D\'epartement d'Informatique, Universit\'e libre de Bruxelles, Belgium}
\email{gwenael.joret@ulb.be}

\author[Petit]{Robin Petit}
\address[R.~Petit]{D\'epartement d'Informatique, Universit\'e libre de Bruxelles, Belgium}
\email{robin.petit@ulb.be}

\author[Raymond]{Jean-Florent Raymond}
\address[J-F.~Raymond]{CNRS, ENS de Lyon, Université Claude Bernard Lyon 1, LIP UMR 5668, Lyon, France}
\email{jean-florent.raymond@cnrs.fr}

\author[Robinson]{Eileen Robinson}
\address[E.~Robinson]{D\'epartement de Mathématiques, Universit\'e libre de Bruxelles, Belgium}
\email{eileen.robinson@ulb.be}

\thanks{W.~Cames van Batenburg, Q.\ Claus and G.\ Joret are supported by the Belgian National Fund for Scientific Research (FNRS)}

\begin{abstract} 
We show that every planar graph with no $t \times t$ grid minor has treewidth at most $4t +4$. 
This improves on the previously best known bound of $\frac{9}{2}t - \frac{11}{2}$, due to Gu and Tamaki (2012), 
and is within a factor $2$ of optimal.  

A key step in the proof is showing the following result, which might be of independent interest: 
Every $2$-connected plane graph $G$ with radius $d$ and faces of size at most $k$ 
has a tree-decomposition of width at most $\max\{3d+ k+5, 2d+2k+1\}$ such that the vertex set of every face of $G$ is contained in some bag. 
\end{abstract}

\maketitle

\section{Introduction}

For every positive integer $t$, let $g(t)$ be the smallest integer such that every planar graph with no $t \times t$ grid minor has treewidth at most $g(t)$. It is well known that $g(t)$ is linear in $t$. The first linear upper bound on $g(t)$ was given by \citet{RST94} in 1994, who proved that $g(t) \leq 6t-5$.
This was improved to $g(t)\leq 5t-6$  by \citet{Grigoriev11}, and then to $g(t) \leq \frac{9}{2}t - \frac{11}{2}$ by \citet{gu2012improved} in 2012, which is the current best upper bound. 

In this paper, we give an improved upper bound of $g(t) \leq 4t +4$. 

\begin{theorem} \label{thm:main}
Every planar graph with no $t \times t$ grid minor has treewidth at most $4t+4$.
\end{theorem}

As for lower bounds, it is known that
$g(t)\geq 2t-3$, as proved by \citet{gu2012improved} via the Cartesian product of a cycle and a path; see also~\cite{aidun2020treewidth}.\footnote{In~\cite{GrigorievEtAl12}, it is mentioned that a construction due to Mazoit and Todinca might be able to show that $g(t) \geq 3t$. However, it turns out that the construction does not give a lower bound better than $g(t)\geq 2t-3$; we thank Ioan Todinca for helpful discussions on this topic.}  Thus our upper bound is (roughly) within a factor $2$ of optimal. 

Our proof of \cref{thm:main} follows closely that of \citet{gu2012improved} for the bound $g(t)\leq \frac{9}{2}t - \frac{11}{2}$ but with some key changes in how the tree-decomposition is being built. In the rest of this introduction, we first review the approach of \citet{gu2012improved}  and then we explain  how we modified it to prove  \cref{thm:main}. 

Given positive integers $k, h$ with $k\geq 3$ and $h\geq 1$, an \defin{\cgrid{k}{h}} 
is the graph obtained by taking the Cartesian product of a cycle on $k$ vertices with a path on $h$ vertices. The main result of  \citet{gu2012improved}  is the following upper bound on the branchwidth of planar graphs having no \cgrid{k}{h} minor. 

\begin{theorem}[\citet{gu2012improved}] \label{thm:GT-cylindrical-grids}
Every planar graph with no \cgrid{k}{h} minor has branchwidth at most $k+2h-3$. 
\end{theorem}

The treewidth and branchwidth of planar graphs are closely related to each other:  
\citet{RS91} showed that if a planar graph has treewidth $w$ and branchwidth $b \geq 2$ then 
\begin{equation}
    \label{eq:tw_vs_bw}
    b-1 \leq w \leq \frac{3}{2}b-1. 
\end{equation}
The aforementioned bound of $g(t) \leq \frac{9}{2}t - \frac{11}{2}$ follows then from the second inequality above combined with \cref{thm:GT-cylindrical-grids} and the fact that the $t \times t$ grid is a subgraph of the \cgrid{t}{t}.  

It is known that both the lower and upper bounds in \eqref{eq:tw_vs_bw} are tight, that is, there exist planar graphs with branchwidth $b\geq 2$ that have treewidth $b-1$, and there also exist some  
that have treewidth $\frac{3}{2}b-1$.
One might hope that by going over the proof of \cref{thm:GT-cylindrical-grids} given in \cite{gu2012improved} and translating the proof from the language of branch-decompositions to that of tree-decompositions, one could prove directly that 
$g(t) \leq 4t + O(1)$ by tweaking the resulting tree-decompositions a bit. 
However, we were unable to do so, which prompted us to develop another approach. 

We proceed with a very informal outline of the proof of \citet{gu2012improved} and then explain the changes that we made to it. 
The proof is split into two main steps: Step 1 is about bounding the branchwidth of planar graphs with small radius. Step 2 is a global inductive proof, separating the given planar graph $G$ in a shallow part and multiple deep parts, and using Step 1 as a black box to bound the branchwidth of the shallow part, and the induction hypothesis to bound the branchwidth of each of the deep parts, and then combining all the resulting branch-decompositions together (at no extra cost) to get the desired bound on the branchwidth of $G$. 
The resulting bound on the branchwidth is governed by the bound from Step 1. 
In this sense, Step 1 is the real \say{engine} of the proof; 
in particular, establishing better bounds for Step 1 would result in better bounds overall.  

Our proof of \cref{thm:main} follows the same general approach, but formulated in terms of tree-decompositions. 
Our Step 2, which is described in \cref{sec:GridMinor}, is conceptually the same as that of \citet{gu2012improved}, up to some minor modifications due to technicalities coming from our setup. 
Our version of Step 1 (see \cref{sec:tree_decompositions}) is where the proof differs, and where new ideas are introduced, so 
let us focus on Step 1. 
It is well-known that planar graphs with small radius have small treewidth, specifically \citet{GM3} (see also \citet{E00}) proved that every planar graph with radius $h$ has treewidth at most $3h + 1$. 
In Step 1, we prove a variant of this result:  
Informally, we are given a plane graph $G$ with radius at most $h$, and a collection of {\em special faces} of $G$, each of size at most $k$. 
The goal is to build a tree-decomposition of $G$ of small width with the extra property that, for each special face, there is a bag of the tree-decomposition containing all the vertices on the boundary of the face. 
By translating Step 1 from the proof in \cite{gu2012improved} (see Section 3 in that paper) to tree-decompositions, 
one can obtain such a tree-decomposition of $G$ that has width at most $3h + \frac{3}{2}k + O(1)$. 
In our version of Step 1, we achieve a width of $3h + k + O(1)$, which allows us eventually to shave off  a $\frac{1}{2}t$ term (roughly) in the resulting upper bound on $g(t)$ when combined with Step 2.  
This is done by building the tree-decomposition differently; in particular, we use a careful necklace splitting argument at a key step of the proof to avoid an extra $\frac{1}{2}k$ cost in the width. 

We conclude this introduction with two remarks. 
First, similarly as for previous upper bounds on $g(t)$ given in the literature, our proof of \cref{thm:main} is algorithmic, it can be turned into a polynomial-time algorithm finding either a tree-decomposition of width at most $4t+4$, or a $t\times t$ grid minor in the plane graph given in input. 
Second, we tried to simplify the exposition of the proof as much as possible and as a result, we made no efforts to optimize the constant term in the bound $g(t) \leq 4t+4$. 
(For instance, $g(t) \leq 4t+1$ can be shown at the price of some extra technicalities when dealing with the nooses in \cref{sec:GridMinor}.) 
On the other hand, the $4t$ term seems quite difficult to improve. 
In particular, we believe that proving a $(4-\epsilon)t + O(1)$ bound using this framework would require proving a $(3-\epsilon)h + O(1)$ bound on the treewidth of planar graphs with radius $h$, which, as far as we are aware, is an open problem since the 1984 work of \citet{GM3}.

The paper is organized as follows. 
The necessary definitions and notations are given in \cref{sec:prelim}. 
The aforementioned result for Step 1, which is the heart of the proof, is proved in \cref{sec:tree_decompositions}. 
Then Step 2, the global inductive proof, is described in \cref{sec:GridMinor}.

\section{Preliminaries}\label{sec:prelim}

In this paper, \say{graph} always means an undirected, finite, simple graph. 
We also consider multigraphs, where parallel edges and loops are allowed. 
We let $V(G)$ and $E(G)$ denote the vertex and edge sets of a (multi)graph $G$, respectively.

A graph $G$ is \defin{$2$-connected} if $G$ is connected, has at least three vertices, and for every $v\in V(G)$, the graph $G-v$ is connected.  
A \defin{block} of a graph $G$ is an induced subgraph of $G$ that is either $2$-connected or isomorphic to $K_1$ or $K_2$, and is inclusion-wise maximal with this property. A block is said to be \defin{non-trivial} if it is $2$-connected, and \defin{trivial} otherwise. 

The \defin{distance} between two vertices $v, w\in V(G)$ of a multigraph $G$, denoted by $d_{G}(v, w)$, is the number of edges in a shortest path in $G$ between $v$ and $w$, if it exists, and $+\infty$ if $v$ and $w$ are not in the same connected component of $G$.

Given two sets $A,B \subseteq V(G)$, a path $P$ in a multigraph $G$ is called an \defin{$A$--$B$ path} if its vertices can be enumerated as $v_1,\ldots,v_k$ along $P$ with $V(P)\cap A = \{v_1\}$ and $V(P)\cap B = \{v_k\}$.  
A subset $X\subseteq V(G)$ \defin{separates} $A$ and $B$ in $G$ if there is no $A$--$B$ path in $G-X$. 

Let $H$ be a graph.  An \defin{$H$-model} in a graph $G$ is a set $\{A_v \ | \ v\in V(H)\}$ of pairwise disjoint nonempty subsets of $V(G)$, each inducing a connected subgraph of $G$, and such that for every edge $uv\in V(H)$, there exists $x\in A_u$ and $y\in A_v$ such that $xy\in E(G)$.
If $G$ contains a $H$-model, we say that $H$ is a \defin{minor} of $G$. 

A \defin{tree-decomposition} of a multigraph $G$ is a pair $(T, \mathcal B)$, where $T$ is a tree, and $\mathcal B:=\{\beta_x\subseteq V(G) \ | \ x\in V(T)\}$ is a collection, such that the subgraph of $T$ induced by $\{x\in V(T) \ | \ v\in \beta_x\}$ is connected for every $v\in V(G)$, and for every edge $uv\in E(G)$ there exists $x\in V(T)$ such that $\beta_x$ contains $u$ and $v$.  The elements of $\mathcal{B}$ are called \defin{bags}. The \defin{width} of $\mathcal B$ is given by $\max_{\beta \in \mathcal B}|\beta| -1$ and the \defin{treewidth} of $G$, denoted by $tw(G)$, is the minimum width of a tree-decomposition of~$G$.

A \defin{rooted tree} is a tree with a specified vertex $r\in V(T)$ called the \defin{root}; 
its \defin{height} is the length of a longest path that has $r$ as an endpoint in $T$. 
For a rooted tree $T$ with root $r$, and a vertex $v\in V(T)$, we call a vertex $w\in V(T)$ a \defin{parent} of $v$ in $T$ if $vw\in E(T)$ and $d_T(r, w)+1=d_T(r, v)$.  It is easy to see that every vertex apart from $r$ has exactly one parent.  We call a vertex $w\in V(T)$ a \defin{child} of $v$ in $T$ if $v$ is the parent of $w$. 
For $x,y\in V(T)$, we define $xTy$ to be the unique path from $x$ to $y$ in $T$.

A \defin{BFS tree} of a connected multigraph $G$ is a rooted spanning tree $T$ of $G$, with root $r\in V(T)$, such that $d_{T}(r, v)=d_G(r, v)$ for every $v\in V(G)$. 

We review some standard notions and terminology about planar graphs, with the corresponding notations used in the paper, sticking to an informal treatment when it comes to the underlying topological aspects. 
We refer the reader to the textbook by~\citet{diestel:graph} for background on this topic. 
A multigraph $G$ is \defin{planar} if there exists a cross-free embedding of $G$ in the plane. 
A \defin{plane multigraph} $G$ is a planar graph together with a specific embedding in the plane where edges do not cross. 
It will be convenient to identify $V(G)$ with the corresponding set of points in the plane, $E(G)$ with the corresponding set of curves in the planes, and see $G$ itself as a subset of the plane.   
By extension, for every subgraph $H$ of $G$, we can also see $H$ as the corresponding subset of the plane. 
We see the \defin{faces} of $G$ as the connected regions of the plane with the drawing of $G$ removed, and we let $F(G)$ denote the set of faces of $G$. 
Given a face $F$ of $G$, we let $G[F]$ denote the subgraph of $G$ corresponding to the boundary of $F$. 
When no confusion can occur, we sometimes treat a face and its corresponding subgraph interchangeably, for simplicity; 
e.g.\ when we say that a face $F$ contains an edge $e$ we mean that $G[F]$ contains~$e$.

The \defin{dual graph} of a plane graph $G$ is the multigraph $G^\star$ such that $V(G^\star)$ corresponds to the set of faces of $G$, and each edge $e$ of $G$ gives rise to an edge of $G^{\star}$ (called \defin{dual edge} of $e$) joining the vertices associated with its incident faces. 

Note that if  $G$ is a triangulation with at least $4$ vertices then $G^{\star}$ is a simple  graph. 
The \defin{radial graph} of $G$, denoted by $\mathcal{R}_G$, is the graph such that $V(\mathcal{R}_G)=V(G)\cup F(G)$, and two vertices are adjacent in $\mathcal{R}_G$ if and only if they correspond in $G$ to a vertex and a face containing this vertex.  
It will be convenient to have a plane embedding of $\mathcal{R}_G$, and thus we assume that $\mathcal{R}_G$ is drawn in the plane, with the vertices in $V(G)$ at their respective positions in $G$, the vertices in $F(G)$ inside their respective face, and edges drawn in the natural way.

A \defin{noose} of a plane graph $G$ is a cycle in $\mathcal{R}_G$. 
We remark that, in the literature, a noose is traditionally defined as a closed curve in the plane intersecting 
$G$ only in vertices and crossing each face of $G$ at most once.  
When $G$ is $2$-connected, the two definitions are essentially equivalent; however, our definition ensures that we consider only finitely 
many nooses for $G$, which simplifies some of the arguments later on. 
(We remark that we will only consider radial graphs of $2$-connected graphs in the paper.) 
Given a noose $N$, we will often need to consider the set of vertices of $G$ that appear on $N$; we denote this set by $V_G(N)$. 
Two vertices of $V_G(N)$ are said to be \defin{consecutive} on the noose $N$ if they are at distance $2$ in the cycle $N$ of $\mathcal{R}_G$ (and thus appear on a common face). 

Let $\gamma$ be a closed curve in the plane.  
We let $\interior(\gamma)$ and $\Interior(\gamma)$ denote the open disk and the closed disk, respectively, bounded by $\gamma$. 
These notations will often be used for $\gamma$ being the curve defined by a noose $N$ of a plane graph $G$.

A closed curve $\gamma$ in the plane \defin{separates} two subgraphs $G_1, G_2$ of a plane graph $G$ if, up to exchanging $G_1, G_2$ if necessary, 
$G_1$ is contained in $\Interior(\gamma)$ and $G_2$ is contained in $\R^2 \setminus \interior(\gamma)$. 
Similarly, given $Y_1, Y_2 \subseteq V(G)$, we say that $\gamma$ separates $Y_1$ from $Y_2$ if, up to exchanging $Y_1, Y_2$, $Y_1$ is contained in $\Interior(\gamma)$ and $Y_2$ is contained in $\R^2 \setminus \interior(\gamma)$. 

For a plane graph $G$ and a spanning tree $T$ of $G$, the \defin{co-tree} of $T$ is the subgraph $T^\star$ of the dual $G^\star$ such that the set of vertices of $T^\star$ is the set of faces of $G$, and the set of edges of $T^\star$ is the set of edges of $G^\star$ that are not dual to any edge of $T$.  It is well-known that this graph is a spanning tree of the dual of $G$ (hence the name).

\section{Tree-decompositions of plane graphs accommodating their faces}
\label{sec:tree_decompositions}

The following theorem is the main result of this section.  

\begin{theorem}
\label{thm:td_faces}
    Let $G$ be a $2$-connected loopless plane multigraph, let $R$ be a face of $G$, and let $d$ be the maximum distance in $G$ between $V(G[R])$ and a vertex of $G$.
    Let $k$ be the maximum number of vertices in a face of $G$. 
    Then $G$ admits a tree-decomposition of width at most $\max\{3d + k + 5, 2d + 2k +1\}$ such that the vertex set of every face of $G$ is contained in some bag.
\end{theorem}

\begin{proof}
With a slight abuse of notation, we let $G_0$ denote the plane multigraph $G$ from the statement, and we let $G$ denote the plane multigraph obtained from $G_0$ by adding a new vertex $c(R)$ inside the face $R$ and making it adjacent to all vertices in $V(G_0[R])$. 
The rest of the proof will be focused on this graph, which is why we prefer to use the notation $G$ for it. 
Denote the newly created triangular faces by $R_1, R_2, \dots, R_\ell$ in a cyclic order around $c(R)$, where $\ell \coloneqq \card{V(G_0[R])}$. 
Let $\mathcal{R}=\{R_1, \ldots, R_\ell\}$ and let $\mathcal{F}$ be the set of faces of $G$ that are not in $\mathcal{R}$.  

Let $G'$ be the multigraph obtained from $G$ by adding, for every $F\in \mathcal{F}$, a new vertex $c(F)$ inside the face $F$ and making it adjacent to all vertices in $V(G[F])$. Observe that compared to the radial graph of $G$, the graph $G'$ also contains the edges of $G$.

Let $T$ be a BFS tree of $G$ rooted at $c(R)$, and let $h$ be its height. 
Observe that $h = d+1$.

The proof is split into four main steps: First, we extend $T$ into a (carefully chosen) spanning tree $T'$ of $G'$. 
Then, we modify further $T'$ into a tree $T^+$ and use it to construct a pair $(T^+, \mathcal{B})$ of a tree and a set of bags. 
Next, we show that $(T^+, \mathcal{B})$ is a tree-decomposition of $G'$.  
Finally, we prove that its width is as desired and that the vertex set of every face of $G_0$ is contained in some bag.

\textbf{Step 1: Defining the spanning tree $T'$ of $G'$.} 

Let $T^\star$ be the co-tree of $T$ w.r.t.\ $G$, which is thus a spanning tree of the dual of $G$. 
Observe that the faces $R_1, R_2, \dots, R_\ell$ are leaves of $T^\star$. 
To help the description of the arguments below, it will be convenient to color these leaves of $T^\star$ in red. 
For every subtree of $T^\star$, we define a corresponding {\em weight}, which is the number of red vertices it contains. 
We define an orientation of the edges of $T^\star$ according to these weights as follows. For every edge $xy$ of $T^\star$, orient $xy$ towards $x$ if the component of $T^\star - xy$ containing $x$ has weight 
larger than that of the component containing $y$, otherwise orient it towards $y$.
In case the two subtrees both have weight equal to $\ell/2$, choose an arbitrary orientation of the edge $xy$. 
Since $T^\star$ is a tree, there is a sink in that orientation; let $s$ denote a sink of the orientation. 
Let $S$ be the corresponding face in $G$.

\begin{claim}
$S$ is not incident to $c(R)$. 
\end{claim}
\begin{proofclaim}
Assume for contradiction that $S$ is incident to $c(R)$.     
Then, $s$ itself is a red leaf in $T^\star$.    
Let $e$ be the edge of $T^\star$ that is adjacent to $s$.  The two connected components of $T^\star-e$ are $T^\star[\{s\}]$, and $T^\star-s$.  Since $s$ is a sink of our orientation, 
the weight of $T^\star[\{s\}]$ should be at least that of $T^\star-s$. 
However, since $T^\star[\{s\}]$ has weight $1$, this implies that 
$T^\star - s$ contains at most one red leaf, and thus $T^\star$ contains at most two red leaves;  
that is, $\ell \leq 2$, which contradicts the fact that $\ell \geq 3$ implied by the definition of $\ell$.    
\end{proofclaim}

Next, we assign a weight $w(v)$ to every vertex $v$ of $G$ and a weight $w(e)$ to every edge $e$ of $G$. 
This is different from the weights defined above for subtrees of $T^\star$, though the two notions are related. 
For every vertex $v$ of $G$, let $w(v) \coloneqq 1$. 
For every edge $e$ of $G$ in $T$, let $w(e) := 0$. 
For every edge $e$ not in $T$, let $w(e)$ be the minimum weight of the two subtrees of $T^\star - e^\star$, where $e^\star$ denotes the dual edge of $e$.
Finally, for every submultigraph $H$ of $G$, we define its weight $w(H)$ to be the sum of the weights of its vertices and edges. 

Note that the boundary $G[F]$ of every face $F$ of $G$ is a cycle of $G$, since $G$ is $2$-connected. 

\begin{claim}\label{cl:weights}
The following holds: 
    \begin{enumerate}
        \item $w(e) \leq \floor{\frac \ell 2}$ for every edge $e$ of $G$,
        \item \label{it:sumz} $\sum_{e\in E(G[S])} w(e) = \ell$, and
        \item $w(G[S]) = \card{G[S]} + \sum_{e\in E(G[S])} w(e) \leq k + \ell \leq 2k$. 
    \end{enumerate}
\end{claim}
\begin{proof}
    Let $e$ be an edge of $G$.  If $e\in E(T)$, then $w(e)=0$, and we are done.  Now, assume that $e\not \in E(T)$. Then the dual edge $e^\star$ belongs to $T^\star$. Recall that the weights of the two subtrees of $T^\star-e^\star$ are integers that sum to $\ell$, so the minimum of these weights is $\floor{\frac \ell 2}$.
    This shows the first property.

    For every edge $e$ of $G[S]$, either $e\in E(T)$ in which case $w(e)=0$, or $e\notin E(T)$ and then the choice of $S$ ensures that $w(e)$ is defined as the number of red leaves of the subtree of $T^\star-e^\star$ not containing $s$.
    As these subtrees are disjoint, we get $\sum_{e\in E(G[S])} w(e) \leq \ell$. Equality follows from the fact that these subtrees contain all the red leaves of $T^\star$.
    
    The third property of the statement directly follows from the second one and the fact that $S$ has at most $k$ vertices. 
\end{proof}

Given two distinct vertices $u, v$ of the cycle $G[S]$, we call $\{u,v\}$ a \defin{splitting pair} if both paths from $u$ to $v$ on $G[S]$ have weight at most $k+2$. 

\begin{claim} 
\label{claim:splitting_pair}
    The cycle $G[S]$ has a splitting pair.
\end{claim}

\begin{proofclaim} 
    We rephrase the problem as a necklace splitting problem. 
    Our necklace $N$ is a cycle with its vertices colored red or blue, obtained as follows: 
    Start with the cycle $G[S]$, color all its vertices blue, then subdivide every edge $e$ of $G[S]$ with $w(e)$ internal vertices, which are all colored red.
    Next, if $\card{G[S]} < k$, then subdivide  $k -|G[S]|$ times an arbitrarily chosen edge $e_b$ incident to a blue vertex $v_b$ and color blue the newly introduced vertices---these vertices are said to \defin{correspond} to $v_b$---and finally, if $\ell < k$, then subdivide $k - \ell$ times an arbitrarily chosen edge $e_r$ incident to a red vertex  and color red the newly introduced vertices.    
    By the fact that $S$ has at most $k$ vertices and item \ref{it:sumz} of Claim~\ref{cl:weights}, the construction above results in a cycle $N$ that has $k$ red vertices and $k$ blue vertices. 

    Observe that, by construction, if $u, v$ are two distinct blue vertices of $N$, then $u, v$ correspond to vertices $u', v'$ of $G[S]$ such that, for each of the two paths $P$ between $u$ and $v$ on $N$, the number of vertices of $P$ is at least the weight $w(P')$ of the corresponding path $P'$ between $u'$ and $v'$ on $G[S]$. 
        
    Since $N$ has even length, every vertex $u$ of $N$ has a well-defined \defin{opposite vertex} $v$ that is at distance exactly $k$ on $N$. 
    (Note that $u$ is then the opposite vertex of $v$.) 
    
    {\bf Case 1: There exists a pair of opposite blue vertices.} 
    Let $u,v$ be such a pair and let $u', v'$ be the corresponding vertices on $G[S]$. 
    Since the two paths between $u$ and $v$ on $N$ both have $k+1$ vertices, the corresponding paths between $u$ and $v$ on $G[S]$ have weight at most $k+1$, and hence $\{u', v'\}$ is the desired splitting pair. 

    {\bf Case 2: Every pair of opposite vertices has at least one red vertex.} 
    Since there are exactly $k$ red vertices and $k$ blue vertices on $N$, and there are $k$ distinct pairs of opposite vertices, it follows that every such pair contains exactly one red vertex and one blue vertex.     
    Let $\widetilde{u}, u$ be two vertices such that $\widetilde{u}$ is colored red and $u$ is a neighbor of $\widetilde{u}$ that is colored blue.  Let $v$ be the vertex opposed to $\widetilde{u}$.  Thus, $v$ is colored blue.
    Then each of the two paths between $u, v$ on $N$ has at most $k+2$ vertices, and thus the two paths between the vertices $u', v'$ of $G[S]$ corresponding to $u, v$ both have weight at most $k+2$, and hence  $\{u', v'\}$ is the desired splitting pair. 
    \end{proofclaim}

For each face $F\in \mathcal{F}$, we define a corresponding pair $p(F), o(F)$ of vertices of $F$ as follows. 
If $F=S$, we choose these two vertices in such a way that $\{p(F), o(F)\}$ is a splitting pair of $S$. 
If $F\neq S$,  we choose $p(F), o(F)$ so that they are at distance $\left \lfloor |F| / 2\right \rfloor$ on the cycle $G[F]$.

Let $T'$ be the spanning tree of $G'$ obtained from $T$ as follows: For every face $F$ of $G$ considered in the previous paragraph, 
we add the edge $p(F)c(F)$ to $T$. Observe that $T'$ is a rooted tree of height at most $d+2$, and that every vertex $c(F)$ is a leaf of $T'$, with parent $p(F)$. 
(We remark also that the vertex $o(F)$ is roughly opposite to $p(F)$ on the cycle $G[F]$, hence the choice of the letter $o$.) 

\textbf{Step 2. Building a tree-decomposition of $G'$.}

Let $T'^\star$ be the co-tree of $T'$ w.r.t.\ $G'$. 
In order to define the tree indexing our tree-decomposition of $G'$, we need to slightly modify $T'^\star$ as follows: Let $T^+$
be the tree obtained from $T'^\star$ by subdividing once the edge of $T'^\star$ corresponding to the edge $o(F)c(F)$ of $G'$, for every face $F\in \mathcal{F}$; we let $x^+(F)$ denote the newly introduced subdivision vertex and let $\mathcal X^+$ denote the set of these subdivision vertices. Conversely, for every $x\in \mathcal X^+$ we denote by $F(x)$ the face in $\mathcal{F}$ such that $x=x^+(F(x))$.  

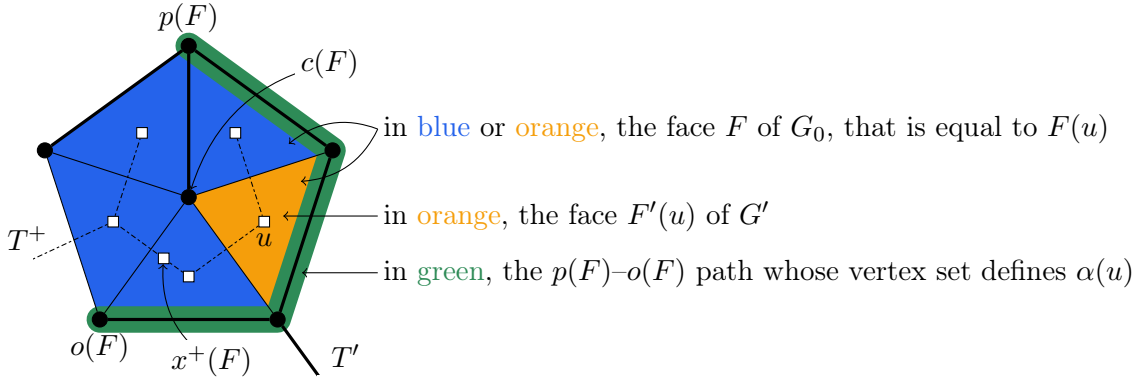
\begin{figure}[h]
    \centering
    \begin{tikzpicture}[
    every node/.style={draw=none, fill=white, inner sep=2pt},
  vertex/.style={
    circle,
    draw,
    fill=black,
    inner sep=2pt
  }, svertex/.style={
    rectangle,
    draw,
    fill=white,
    inner sep=2pt
  }
]

  \definecolor{col1}{HTML}{F59E0B}
  \definecolor{col2}{HTML}{2563EB}
  \definecolor{col3}{HTML}{2e8b57}

  \node[vertex, label=90:$p(F)$] (v1) at (90:2cm)  {};
  \node[vertex] (v2) at (18:2cm)  {};
  \node[vertex] (v3) at (-54:2cm) {};
  \node[vertex, label=-90:$o(F)$] (v4) at (-126:2cm) {};
  \node[vertex] (v5) at (162:2cm) {};

  \node[vertex] (c) at (0,0) {};

  \node[svertex] (f12) at (54:1.05cm)  {};
  \node[svertex, label=-90:$u$] (f23) at (-18:1.05cm)  {};
  \node[svertex] (f34) at (-90:1.05cm)  {};
  \node[svertex] (f45) at (-162:1.05cm) {};
  \node[svertex] (f51) at (126:1.05cm)  {};

  \draw[very thick] (v5) -- (v1) -- (v2) -- (v3) -- (v4);
  \draw (v4) -- (v5);

  \draw[very thick] (v3) -- ++(-54:1cm) node[draw=none, label=70:$T'$]{};
  
  \draw[very thick] (c) -- (v1);
  \draw (c) -- (v2)
        (c) -- (v3)
        (c) -- (v4)
        (c) -- (v5);
  
  \begin{scope}[every path/.style = {dash pattern=on 1pt off 1pt on 2pt off 1pt}]
    \draw (f12) -- (f23) -- (f34);
    \draw (f34) -- node[pos=0.3, solid, svertex] (x+) {} (f45);
    \draw (f45) -- (f51);
    \draw (f45) --++(-155:1.25cm) node[draw=none, label=$T^+$] {};
  \end{scope}

  \draw[<-] (x+) to[bend right] ++(-65:1.5cm) node {$x^+(F)$};
  \draw[<-] (c)[xshift=0.1cm] to[bend left] (45:2.5cm) node {$c(F)$};
  \draw[<-] (1.25, -0.25) to (2.5, -0.25) node[anchor = west] {in \textcolor{col1}{orange}, the face $F'(u)$ of $G'$};
  \draw (2.5cm, 0.9cm) node[anchor = west] (lvo) {in \textcolor{col2}{blue} or \textcolor{col1}{orange}, the face $F$ of $G_0$, that is equal to $F(u)$};
  \draw[<-] (28:1.5cm) to[bend left] (lvo.180);
  \draw[<-] (8:1.5cm) to[bend right] (lvo.180);
  
  \draw[<-] (1.5, -1) to (2.5, -1) node[anchor = west] {in \textcolor{col3}{green}, the $p(F)$--$o(F)$ path whose vertex set defines $\alpha(u)$};
  
  \begin{scope}[on background layer]
    \fill[col2] (v1.center) -- (v2.center) -- (v3.center) -- (v4.center) -- (v5.center) -- cycle;
    \fill[col1] (v2.center) -- (c.center) -- (v3.center) -- cycle;
    \draw[line width = 0.35cm, line join=round, line cap=round, color = col3] (v1.center) -- (v2.center) -- (v3.center) -- (v4.center); 
  \end{scope}
\end{tikzpicture}
    \caption{The different notations introduced in Step 2.}
    \label{fig:face}
\end{figure}

\begin{figure}[h]
    \centering
    \includegraphics[width=01\linewidth]{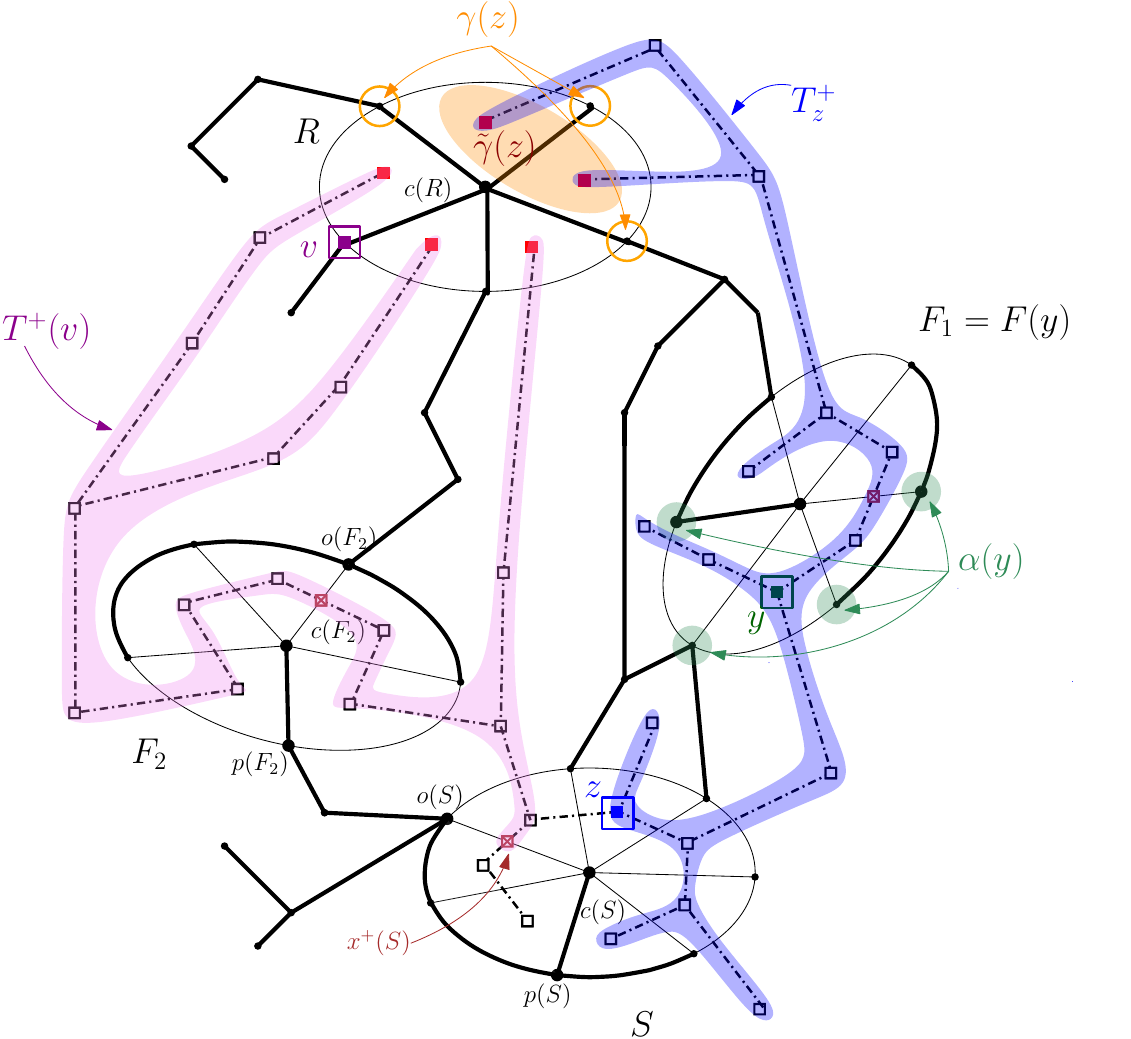}
    \caption{This figure summarizes all the notations introduced in Step 1 and Step 2.  This is a partial drawing of $G'$ together with its spanning tree $T'$ (heavy black edges) and $T^+$, the extension of the dual of $T'$ (squared vertices and dash dotted edges). The vertices of $T^+$ corresponding to $\{R_1,\ldots,R_\ell\}$ are colored in red.}
    \label{fig:t+x}
\end{figure}

Observe that every $x\in V(T^+) \setminus \mathcal X^+$ corresponds to a (triangular) face of $G'$, that we call $F'(x)$. Also, $F'(x)$ contains a unique vertex of the form $c(F)$ for some face $F$ of $G_0$. We call this face $F(x)$.  
In case $F(x)$ is not $R$, we also define $\alpha(x)$ as follows: Consider the (unique) $p(F(x))$--$o(F(x))$ path on the cycle $G[F(x)]$ that contains an edge of $G'[F'(x)]$, and let $\alpha(x)$ denote the vertex set of that path. 

In the rest of the proof it will be convenient to see $T^+$ as a rooted tree, rooted at $x^+(S)$; given a vertex $x$ of $T^+$, the \defin{subtree $T^+_x$} is the subtree of $T^+$ spanned by $x$ and all its descendants in $T^+$. 
For every vertex $x\in V(T^+)$, we define $\widetilde{\gamma}(x)$ as the set of red vertices (i.e.\ the set of vertices of $T^+$ corresponding to the faces $R_1, \ldots, R_\ell$) in the subtree $T^+_x$, and $\gamma(x)$ as the set of vertices of $G_0$ in faces corresponding to vertices in $\widetilde{\gamma}(x)$. We highlight the fact that $c(R)\not \in \gamma(x)$.  

We define a tree-decomposition $(T^+, \mathcal{B})$ of $G'$ indexed by the tree $T^+$ as follows, where $\mathcal{B}=\{\beta_x \mid x\in V(T^+)\}$ is the set of bags and 
\[\beta_x \coloneqq \begin{cases}
    \gamma(x) \cup \left(\bigcup_{u \in V(G'[F'(x)])}  V\left(uT'c(R)\right)\right) \cup \alpha(x) & \text{ if $x\notin \mathcal{X}^+$ and $F(x)\in \mathcal F$} \\[2ex] 
    \gamma(x)\cup\left ( \bigcup_{u \in V(G'[F'(x)])} V\left(uT'c(R)\right ) \right )  & \text{ if $x\notin \mathcal{X}^+$ and  $F(x)=R$}\\[2ex]
    \gamma(x) \cup 
    \left(\bigcup_{u \in \{c(F(x)),o(F(x))\}}  V\left(uT'c(R)\right)\right)
    \cup V\left(G[F(x)]\right)& \text{ if $x\in \mathcal{X}^+$}. 
    
\end{cases}\]

\textbf{Step 3. Proving that $(T^+, \mathcal B)$ is a tree-decomposition of $G'$.} 

To prove that $(T^+, \mathcal{B})$ is a tree-decomposition of $G'$, we first define another, simpler, tree-decomposition $(T'^\star, \mathcal{B}')$ of $G'$, indexed by the co-tree $T'^\star$ of $T'$ w.r.t.\ $G'$, where $\mathcal{B}'=\{\beta'_x \mid x\in V(T'^\star)\}$ and the bags are defined as follows: 
For each $x\in V(T'^\star)$, 
\[
\beta'_x \coloneqq \bigcup_{u \in V(G'[F'(x)])}  V\left(uT'c(R)\right). 
\]
First, we show that $(T'^\star, \mathcal{B}')$ is indeed a tree-decomposition of $G'$. 
Then, by considering how $(T^+, \mathcal{B})$ can be obtained by modifying $(T'^\star, \mathcal{B}')$, we will deduce that $(T^+, \mathcal{B})$ is also a tree-decomposition of $G'$. 
Note that $(T'^\star, \mathcal{B}')$ is already a tree decomposition of $G'$ (and hence of the subgraph $G_0$) with bags of sufficiently small size $(\le 3d+O(1)$) such that the vertex set of each face of $G'$ is contained in some bag. However, we do not yet have the same property for every face of the original graph $G_0$, and that is why we need the further modification of $(T'^\star, \mathcal{B}')$ to $(T^+,\mathcal{B})$.

Since the vertex set of each face of $G'$ is contained in some bag of $\mathcal{B}'$, in particular we have that every vertex and every edge of $G'$ is included in some bag of $\mathcal{B}'$. 
To show that $(T'^\star, \mathcal{B}')$ is a tree-decomposition of $G'$, it remains to show that the vertices of $T'^\star$ whose bags contain a fixed vertex of $G'$ span a connected subgraph of $T'^\star$, which is a consequence of the following claim.

\begin{claim}\label{tocheck: quentin}
Let $x_1 x_2 \dots x_m$ be a path in $T'^\star$. If some $v\in V(G')$ belongs to both $\beta'_{x_1}$ and $\beta'_{x_m}$, then it belongs to $\beta'_{x_i}$ for every $i\in [m]$.     
\end{claim}
\begin{proofclaim}
Arguing by contradiction, suppose that  $v\notin \beta'_{x_i}$ for some $i\in [m]$, and let $j$ be a smallest such index $i$. 
Note that $2 \leq j \leq m-1$. 
Let $ab$ denote the edge of $G'$ dual to $x_{j-1}x_j$ of $T'^\star$. 
Note that $ab \notin E(T)$. 

Observe that $a,b\in V(G'[F'(x_{j-1})]) \cap V(G'[F'(x_{j})])$. 
Since $v\notin \beta'_{x_j}$, it follows that 
\begin{equation}
\label{eq:v_not_in_two_paths}
v\notin V(aT'c(R)) \cup V(bT'c(R)).     
\end{equation}

Let $C \coloneqq aT'b + ab$. 
Then $C$ is a cycle of $G'$ that avoids $v$. 
Seeing $C$ as a plane graph, let $Q_1, Q_2$ denote the two corresponding faces of its drawing. 
Let us emphasize that $ab$ is the only edge of $C$ that is not in $T$. 
Also, one of the two faces $F'(x_{j-1}), F'(x_{j})$ is contained in $Q_1$ and the other in $Q_2$, 
say without loss of generality $F'(x_{j-1})$ is in $Q_1$ and $F'(x_{j})$ in $Q_2$. 
By planarity and because the edges of the path $x_1\dots x_m$ can cross the cycle $C$ only once (in $x_{j-1}x_j$, because the other edges of $C$ belong to $T'$), it follows that $F'(x_{i})$ is included in $Q_1$ for every $i\in [j-1]$, 
and that $F'(x_{i})$ is included in $Q_2$ for every $i\in [m] \setminus [j-1]$. 

Now, observe that $v$ is in the interior of $Q_1$ or $Q_2$, since $v$ is not on $C$. 
Reversing the numbering of the vertices of $P$ if necessary, we may assume that $v$ is in the interior of $Q_1$ (while maintaining the assumptions above). 
Since $v\in \beta'_{x_m}$, the vertex $v$ belongs to $V(uT'c(R))$ for some vertex $u$ in $G'[F'(x_{m})]$. 
Since the face $F'(x_{m})$ is contained in $Q_2$, and since $v$ is in the interior of $Q_1$, it follows that the path 
$uT'c(R)$ intersects the cycle $C$. 

Let $u'$ be the first vertex of the path $uT'c(R)$ that is in $C$ starting from $u$ (possibly $u'=u$). 
Then $v$ is in the path $u'T'c(R)$. 
However, since $u'$ is in $C$, and thus in $aT'b$, we also know that $u'$ is in $aT'c(R)$ or in $bT'c(R)$, say $aT'c(R)$. 
Therefore, the path $u'T'c(R)$ is a suffix of the path $aT'c(R)$, and we deduce that $v$ is in the path $aT'c(R)$, contradicting \eqref{eq:v_not_in_two_paths}.
\end{proofclaim}

Now, observe that $(T^+, \mathcal B)$ can be obtained from $(T'^\star, \mathcal B')$ by the following sequence of operations:

\begin{itemize}
    \item For every face $F\in \mathcal F$, we subdivide the edge $x_1x_2$ of $T'^\star$ that is dual to the edge $c(F)o(F)$ of $G'$ and define at first the bag of this new vertex $x^+(F)$ as 
    \[
    V\left(o(F)T'c(R)\right)\cup V\left(c(F)T'c(R)\right).
    \]     
    For $i=1,2$, let $F_i$ be the face of $G'$ corresponding to $x_i$, and let $v_i$ be the vertex of $G'[F_i]$ that is not in $\{o(F),c(F)\}$. 
    Observe that, by planarity, 
    \[
    V\left(v_1T'c(R)\right)\cap V\left(v_2T'c(R)\right)\subseteq V\left(o(F)T'c(R)\right)\cup V\left(c(F)T'c(R)\right).
    \]
    It follows that 
    \[
    V\left(o(F)T'c(R)\right)\cup V\left(c(F)T'c(R)\right) = \beta'_{x_1} \cap \beta'_{x_2}, 
    \]
    that is, the bag of $x^+(F)$ is equal to the intersection of the bags of its two neighbors $x_1, x_2$. 
    It is clear that the result is still a tree-decomposition of $G'$.
    
    \item For every vertex $v$ of $G_0$ and every face $F\in \mathcal{F}$ with $v\in V(G_0[F])$, we proceed as follows: 
    Let $P$ be the subgraph of $T^+$ induced by $x^+(F)$ and the vertices of $T^+$ corresponding to faces of $G'$ incident to $c(F)$. 
    Observe that $P$ is a path. 
    We will add $v$ to the bags of either all vertices of $P$, or to one of the two \say{halves} of $P$ w.r.t.\ vertex $x^+(F)$. 
    \begin{itemize}
    \item If $v=p(F)$, then observe that $v$ is already in the bags of all vertices in $P$, since $v$ is on the $c(F)T'c(R)$ path.     
    \item If $v=o(F)$, then $v$ is already in the bags of $x^+(F)$ and of its two neighbors on $P$; we add $v$ to every bag of vertices in $P$. 
    \item If $v\notin\{o(F), p(F)\}$, then enumerate the vertices of $P$ in order as $x_1, \dots, x_m, x^+(F), y_1, \dots, y_q$, in such a way that $v$ is incident to a face $F'$ of $G'$ corresponding to $x_i$ for some $i\in [m]$ (note that there are two such faces). Observe that $v$ is already in the bag of $x_i$. We add $v$ to all the bags of $x_1, \dots, x_m, x^+(F)$. 
    \end{itemize}
    Again, it is clear that the result is still a tree-decomposition of $G'$. 
    As a result of these modifications, $V(G[F])$ is now included in the bag of $x^+(F)$ for every $F\in \mathcal{F}$; 
    also, $\alpha(x)$ is now included in the bag of $x$ for every $x\in V(T^+)$ with  $x \notin \mathcal{X}^+$ and $F(x) \in \mathcal{F}$.  
    \item For every vertex $v\in V(G_0[R])\cup \{c(R)\}$, we add $v$ to the bag of every vertex of 
    \[
    T^+(v)\coloneqq T^+[\{x\in V(T^+)\mid v\in \gamma (x)\}]. 
    \]
    Let $i \in [\ell]$ be such that $v$ is incident to the face $R_i$ of $G'$, and let $x$ be the vertex of $T^+$ corresponding to $R_i$. 
    Observe that $v$ was already in the bag of $x$, and that $v\in \gamma(x)$, and hence $x$ is in $T^+(v)$.  
    Thus, in order to show that the operation above results in a tree-decomposition of $G'$, it suffices to show that $T^+(v)$  is connected. This is shown in \cref{clm: add_gamma} hereunder. 
\end{itemize}

\begin{claim} \label{clm: add_gamma}
    Let $v\in V(G_0[R])\cup \{c(R)\}$.  Then, $T^+(v)$ is connected. 
\end{claim}
\begin{proofclaim}
   If $v=c(R)$, then $T^+(v)$ is the empty graph, and we are done.  Now, suppose that $v\in V(G_0[R])$.  Let $i, j$ be such that $R_i$ and $R_j$ are the two faces of $G'$ whose boundaries contain the edge $c(R)v$.   
   Let $r_i, r_j\in V(T^+)$ be the vertices corresponding to $R_i$ and $R_j$, respectively.  

Let $x\in V(T^+)$ be such that $v\in \gamma(x)$.
By definition of $\gamma$, there exists a face $F'\in \mathcal R$ such that $v\in V(G[F'])$, and the vertex $r \in V(T^+)$ corresponding to $F'$ is such that $r\in \widetilde{\gamma}(x)$.
However, the two only possible choices for such an $r$ are $r=r_i$ and $r=r_j$. 
It follows that 
\[
v\in \gamma(x) \; \iff \; r_i\in \widetilde{\gamma}(x) \textrm{ or } r_j\in \widetilde{\gamma}(x).
\]
Thus, $T^+(v)= T^+[\{x\in V(T^+)\mid r_i\in \widetilde{\gamma}(x) \lor r_j\in \widetilde{\gamma}(x)\}]$.

By definition $\{x\in V(T^+) \mid r_i\in \widetilde{\gamma}(x)\}$ is the set of ancestors of $r_i$ in $T^+$ (recall that $T^+$ is rooted at $x^+(S)$), and the same goes for $r_j$.
Thus, $T^+(v)$ is the subtree of $T^+$ induced by the vertex sets of two paths both having the root as an endpoint, and thus is connected.  
\end{proofclaim}

This concludes the proof that $(T^+,\mathcal B)$ is a tree-decomposition of $G'$.

\textbf{Step 4. Proving that $(T^+, \mathcal{B})$ has the desired properties.} 

First, observe that $V(R)\subseteq \beta_{x^+(S)}$ (because $V(R)=\gamma(x^+(S))$) and that, for every face $F$ of $\mathcal{F}$, we have $V(G[F])\subseteq \beta_{x^+(F)}$.
Hence, the vertex set of every face of $G_0$ is contained in some bag of $\mathcal{B}$. 

Next, we bound the size of the bags. 
\begin{claim}\label{claim:bag_normal}
    $\left|\bigcup_{u\in V(G'[F'(x)])}V\left(uT'c(R)\right)\right|\leq 3d+5$ for every $x\in V(T^+)\setminus \mathcal X^+$.
\end{claim}

\begin{proofclaim}
    Let $x\in V(T^+)\setminus \mathcal X^+$. 
    Then the face $F'(x)$ of $G'$ is incident to three vertices, $c(F(x))$ and two vertices $v_1, v_2$ of $G$.  
    The vertices $v_1, v_2$ are at distance at most $d+1$ from $c(R)$ in $T'$, and $c(F(x))$ is at distance at most $d+2$ from $c(R)$ in $T'$. 
    Since the three paths $v_1T'c(R)$, $v_2T'c(R)$, $c(F(x))T'c(R)$ all have $c(R)$ in common, we deduce that the union of these three paths has at most $3d+5$ vertices.      
\end{proofclaim}

\begin{claim}\label{claim:consecutive}
$|\gamma(x)|\leq |\widetilde{\gamma}(x)|+1$ for every $x\in V(T^+)$ such that $F(x)\ne S$. 
\end{claim}
\begin{proofclaim}
Let $x\in V(T^+)$ with $F(x)\ne S$. 
Observe that, by planarity, the red vertices in $\widetilde{\gamma}(x)$ correspond to consecutive faces around $c(R)$.  
It follows that $|\gamma(x)|\leq |\widetilde{\gamma}(x)|+1$. 
\end{proofclaim}

\begin{claim}\label{claim:size_gamma}
$|\gamma(x)|\leq \lfloor \frac k 2 \rfloor +1$ for every $x\in V(T^+)$ such that $F(x)\ne S$. 
\end{claim}
\begin{proofclaim}
Let $x\in V(T^+)$ be such that $F(x)\ne S$ and let $y\in V(T^+)$ be the closest ancestor of $x$ in $T^+$ such that $F(y)=S$. 
    Then, $\widetilde{\gamma}(x)$ is included in the set of red leaves of $T^+_y$. 
    Let $e$ be the unique edge of $G'[F'(y)]$ that belongs to $S$.
    Observe that the number of red leaves of $T^+_y$ is exactly the weight of $e$.
    Since $e$ is an edge of $G[S]$, we already observed that $w(e)\leq \lfloor \frac \ell 2 \rfloor\leq \lfloor \frac k 2 \rfloor$.
    Therefore, by definition of $\gamma$ and by \cref{claim:consecutive}, we have $|\gamma(x)|\leq |\widetilde{\gamma}(x)|+1\leq \lfloor \frac k 2 \rfloor +1$. 
\end{proofclaim}

\begin{claim}\label{clm:bounding_outside_of_X+}
    For every $x\in V(T^+)\setminus \mathcal X^+$, 
    \[
    |\beta_x|\leq
    \begin{cases}
         3d+\frac k 2 +6 & \textrm{ if } F(x) = R, \\  
         3d+k+8 & \textrm{ if } F(x)\in \mathcal F. 
    \end{cases}
    \]    
\end{claim}
\begin{proofclaim}
Let $x\in V(T^+)\setminus \mathcal X^+$. 
If $F(x) = R$, then by \cref{claim:bag_normal} and \cref{claim:size_gamma}, 
    \[\card{\beta_x}
    \leq \card{\gamma(x)} + \card{\bigcup_{u \in V(G'[F'(x)])} V\left(uT'c(R)\right)}
    \leq \frac{k}{2} + 1 + 3d + 5
    = 3d +\frac{k}{2} + 6.\]

    Now suppose that $F(x) \in \mathcal F$. 
    By \cref{claim:bag_normal}, we already know that 
    \[
    \card{\bigcup_{u \in V(G'[F'(x)])} V\left(uT'c(R)\right)} \leq 3d + 5
    \]
    Thus, to prove that $|\beta_x|\leq 3d+k+8$, it is enough to show that $\card{\alpha(x)} + \card{\gamma(x)} \le k + 3$, which we do now. 

    If $F(x) \ne S$, then $\alpha(x)$ is the vertex set of a path of length at most $\left \lceil \frac k 2 \right \rceil$, and thus $|\alpha(x)|\leq \left \lceil \frac k 2 \right \rceil + 1$. 
    Since $|\gamma(x)|\leq \lfloor \frac k 2 \rfloor +1$ by \cref{claim:size_gamma}, 
    it follows that $|\alpha(x)|+|\gamma(x)|\leq k+2 < k + 3$.

    Finally, suppose that $F(x) = S$. 
    Let $e$ be the only edge of $G'[F'(x)]$ that is also an edge of $G$, 
    and let $P$ be the $p(F)$--$o(F)$ path of $G[S]$ containing $e$.     
    Note that 
    \[
    |\alpha (x)|=|V(P)|=\sum_{v\in V(P)} w(v).  
    \]
    Let $P'$ be the subpath of $P$ whose first vertex is $p(S)$ and whose last edge is $e$.
    By definition of $T_x^+$ for $x\in V(T^+)$,
    \[
    |\widetilde{\gamma}(x)| = \sum_{f\in E(P')}w(f)\leq \sum_{f\in E(P)}w(f)
    \] 
    (see \cref{fig:t+x}). 
    Since $\{p(S), o(S)\}$ is a splitting pair, we know that $w(P) \le k+2$.  
    Using \cref{claim:consecutive}, it follows that  
    \[|\alpha(x)|+|\gamma(x)|\leq \sum_{v\in V(P)} w(v)+\sum_{f\in E(P)}w(f) + 1 \leq w(P) + 1 \leq k+3,
    \] 
    as desired. 
\end{proofclaim}

\begin{claim} \label{clm:boundind_X+}
For every $x\in \mathcal X^+$, 
    \[
    |\beta_x|\leq
    \begin{cases}
         2d+k+\lfloor \frac k 2 \rfloor +5 & \textrm{ if } x\neq x^+(S), \\  
         2d+2k+4 & \textrm{ if } x = x^+(S). 
    \end{cases}
    \]        
\end{claim}
\begin{proofclaim}
Let $x\in \mathcal X^+$. 
Recall that
\[\beta_x = \gamma(x)\cup \left(V(c(F(x))T'c(R)) \cup V(o(F(x))T'c(R))\right) \cup V(G[F(x)]).\]
The path $c(F(x))T'c(R)$ contains at most $d+3$ vertices while the path $o(F(x))T'c(R)$ contains at most $d+2$ vertices.  Moreover, they both contain the vertex $c(R)$.
Thus
\[\card{V(c(F(x))T'c(R)) \cup V(o(F(x))T'c(R))} \leq 2d+4.\]
Furthermore, $\card{V(G[F(x)])} \leq k$.

If $x\ne x^+(S)$, then $F(x)\ne S$, and $|\gamma(x)|\leq \lfloor \frac k 2 \rfloor +1$ by \cref{claim:size_gamma}.  
If $x=x^+(S)$, then  $\gamma(x)=V(G_0[R])$, which implies $|\gamma(x)|\leq k$.  The result follows.
\end{proofclaim}

By \cref{clm:bounding_outside_of_X+} and \cref{clm:boundind_X+}, we deduce that $(T^+, \mathcal{B})$ is a tree-decomposition of $G'$ of width at most $\max\{3d+k+8, 2d+2k+4\}-1$.

Let $\mathcal{B}_0 = \{\beta_x \cap V(G_0) \mid x \in V(T^+)\}$.
Since $G_0$ is a submultigraph of $G'$, $(T^+, \mathcal{B}_0)$ is a tree-decomposition of $G_0$.

Observe that the vertex $c(R)$, which is not a vertex of $G_0$, is included in every bag of $\mathcal{B}$, and that, for every $x\in V(T^+)$ such that $F(x)\neq R$, the vertex $c(F(x))$ belongs to $\beta_x$ but not to $V(G_0)$.  Thus, for every $x\in V(T^+)$, we have that if $F(x)\neq R$, the bag $\beta_x$ contains two vertices of $G$ that are not vertices of $G_0$ and, if $F(x)= R$, the bag $\beta_x$ contains one vertex of $G$ that is not a vertex of $G_0$.
This observation, combined with \cref{clm:bounding_outside_of_X+} and \cref{clm:boundind_X+} imply that the width of $(T^+, \mathcal B_0)$ is at most $\max\{3d+k+5, 2d+2k+1\}$, which ends the proof. 
\end{proof}

While \cref{thm:td_faces} gives a tree-decomposition accommodating all faces of a given $2$-connected loopless plane multigraph graph $G$, for our purposes we need a variant where we only accommodate a given subset $\mathcal{F}$ of faces.  
This is handled by the following corollary. 

\begin{cor}
\label{cor:td_faces}
    Let $G$ be a $2$-connected loopless plane multigraph, let $\mathcal{F}$ be a set of faces of $G$, and let $R\in \mathcal{F}$.      
    Let $d$ be the maximum distance in $G$ between $V(R)$ and any vertex of $G$.  
    Suppose that every face in $\mathcal{F}$ has size at most $k$. 
    Then $G$ admits a tree-decomposition of width at most $\max\{3d+ k+5, 2d+2k+1\}$ such that the vertex set of every face of $\mathcal{F}$ is contained in some bag.
\end{cor}
\begin{proof}
    Let $H$ be obtained from $G$ by triangulating arbitrarily every face $F$ of $G$ that has at least four vertices, and is not in $\mathcal{F}$, by repeatedly adding edges inside the face $F$ between non-consecutive vertices of the cycle $G[F]$. (It is well-known that this is always possible.) 
    Observe that $H$ is $2$-connected, every face of $H$ has size at most $k$, and every vertex of $H$ is at distance at most $d$ from $V(R)$ in $H$. 
    Thus, applying \cref{thm:td_faces} on $H$ with face $R$, we obtain a tree-decomposition of $H$, and thus of $G$ as well, of the desired width, such that the vertex set of every face of $H$ is contained in some bag. In particular, this holds for the faces in $\mathcal{F}$. 
\end{proof}

\section{Proof of Main Theorem}
\label{sec:GridMinor}

In order prove \cref{thm:main}, 
we prove the following slightly stronger statement, which helps the induction go through. 

\begin{restatable}[Main Technical Theorem]{theorem}{MainTechnicalTheorem}
\label{thm:main induction}
Let $k, h$ be integers with $k\geq 3$ and $h\geq 2$. 
Let $G$ be a $2$-connected plane graph that does not contain a \cgrid{k}{h} as a minor.  
Then there exists a tree-decomposition of $G$ of width at most $\max\{3h+k+4,2h+2k-1\}$. 
Moreover, if the outer face $F$ of $G$ is such that $|V(F)|\leq k-1$, 
then there exists a tree-decomposition of $G$ of the same width such that $V(F)$ is contained in one of the bags. 
\end{restatable}

\cref{thm:main} follows easily from \cref{thm:main induction}, as we now explain: 
Suppose that $G$ is a planar graph with no $t\times t$ grid minor. 
We may assume $t\geq 3$, since \cref{thm:main} is easily seen to hold for $t=1,2$. 
It is well-known (and easy to show) that the treewidth of a graph is the maximum of the treewidth of its blocks, thus it is enough to show that every block of $G$ has treewidth at most $4 t+4$.  
Each trivial block of $G$ (i.e.\ isomorphic to $K_1$ or $K_2$) has treewidth at most $1$. 
Each non-trivial block $B$ of $G$ has treewidth at most $4 t+4$, by \cref{thm:main induction} with $k=h=t$, since $B$ does not contain a $t\times t$ cylindrical grid as a minor. 
Therefore, \cref{thm:main} holds. 

Our proof of \Cref{thm:main induction} follows the same general approach 
as that of \cite{gu2012improved}, as described in Section 4 of their paper, 
except we use \cref{cor:td_faces} from the previous section to get a better tree-decomposition of a certain shallow part $G'$ of the given plane graph $G$, allowing us to get better bounds on the width of the resulting tree-decomposition of $G$. 
While the general ideas at the same, we nevertheless need to slightly adapt the framework of \cite{gu2012improved} to our setting. 
Thus, we describe the proof in full in this section.

\subsection{Family of well-behaved nooses}\label{ssec:Menger}

The goal of this section is to prove that if our plane graph $G$ contains no \cgrid{k}{h} as a minor, then there exists a family of small nooses separating the parts of $G$ that are \say{far away} (depending on $h$) from the outer face of $G$. 
Moreover, this family of nooses behaves well, in the sense that they bound open disks that are pairwise disjoint. 
This heavily relies on results from \citet{gu2012improved}.   
We introduce the necessary definitions and lemmas, and then explain how these lemmas follow, with minor adaptations, 
from lemmas in \cite{gu2012improved}. 

The following lemma gives a refinement of Menger's theorem for plane graphs; it follows from Lemma 4.1 in \cite{gu2012improved}. 

\begin{lemma}[Lemma 4.1 in \cite{gu2012improved}]
\label{lem:planar_Menger}
    Let $G$ be a plane graph, let $C, C'$ be two vertex-disjoint cycles of $G$, and let $k$ be a positive integer.  Then, either there exist $k$ vertex-disjoint $V(C)$--$V(C')$ paths in $G$, or there exists a noose $N$ separating $C$ from $C'$ in $G$ such that $|V_G(N)|\leq k-1$.
\end{lemma}

The next lemma follows from Lemma 4.2 in \cite{gu2012improved}.  

\begin{lem}[Corollary of Lemma 4.2 in \cite{gu2012improved}]\label{lem:cylinder_minor}
    Let $G$ be a plane graph and let $k,h$ be two integers with $k\geq 3$ and $h\geq 2$. 
    Let $C$ and $C'$ be two cycles of $G$ such that $d_{\mathcal{R}_G}(V(C), V(C'))\geq 2(h-1)$. Then,  either $G$ has a \cgrid{k}{h} minor, or there is a noose $N$ in $G$ separating $C$ from $C'$ such that $|V_G(N)|\leq k-1$.  
\end{lem}

Recall that, in a $2$-connected plane graph, every face is bounded by a cycle of the graph. 
In particular, if $B$ is a non-trivial block of a plane graph $G$, then the outer face of $B$ bounds a cycle $C$ of $B$, and we simply say that $B$ is bounded by the cycle $C$. 
Given a plane graph $G$, a face $F$ of $G$, and a nonnegative integer $d$, 
every cycle $C$ of $G$ that bounds a non-trivial block $B$ of $G[\{v\in V(G) \mid d_{\mathcal R_G}(v, F)\geq 2d+1\}]$ 
is called a \defin{contour at depth $d$ from $F$}; the block $B$ is said to be the \defin{block corresponding to $C$}.   

Towards the construction of a \say{well-behaved} family of nooses,  
we need the following lemmas about faces, cycles and contours.

\begin{lem} \label{lem:not_in_a_block}
    Let $G$ be a $2$-connected plane graph.  Let $F$ be its outer face.  Let $d$ be a positive integer, and let $G_d$ be the subgraph of $G$ induced by all the vertices of $G$ at distance at least $2d+1$ from the vertex $F$ in $\mathcal R_G$.  Then, every vertex on the outer face of $G_d$ is at distance exactly $2d+1$ from $F$ in $\mathcal R_G$. 
\end{lem}
\begin{proof}
    First, remark that, since every vertex of $F$ is at distance exactly $1<2d+1$ from $F$ in $\mathcal R_G$, the face $F$ does not share any vertex with the outer face of $G_d$. 
    Let $v$ be a vertex on the outer face of $G_d$.  Since this face is vertex-disjoint from the face $F$ and $G$ is $2$-connected, the vertex $v$ is contained in a face $F_d$ of $G$ that contains at least one vertex $v_d\not \in V(G_d)$.
    
    Since the vertex $v_d$ is not in $V(G_d)$, we have $d_{\mathcal{R}_G}(F, v_d)\leq 2d+1 - 2$ and the following inequalities hold:
    
    \[
    d_{\mathcal{R}_G}(F, v)\leq d_{\mathcal{R}_G}(F, v_d)+d_{\mathcal{R}_G}(v_d, F_d)+d_{\mathcal{R}_G}(F_d, v)\leq 2d-1+1+1=2d+1.
    \qedhere 
    \]
\end{proof}


Let $G$ be a plane graph.  Let $F$ be the outer face of $G$, let $H$ be a $2$-connected subgraph of $G$ that does not share any vertex with $F$, and let $C_1$ be the contour at depth $1$ from $F$ in $G$, such that $H$ is drawn inside the closed disk $\Interior(C_1)$ bounded by $C_1$.  Given a positive integer $k$,
we say that a noose $N$ of $G$ is \defin{$(k, H)$-optimal} in $G$ if $N$ satisfies the following properties:
\begin{enumerate}
    \item $N$ (seen as a closed curve) separates $C_1$ and $H$ in $G$, \label{item:separates}
    \item $|V_G(N)|\leq k-1$ and $|V_G(N)|$ is minimal subject to the first  condition, and
    \item $\interior(N)$ is inclusion-wise minimal subject to the first two conditions. 
\end{enumerate}

Observe that, in particular, every $(k, H)$-optimal noose $N$ separates $H$ and $G[F]$ in $G$, and is such that $V_G(N)\cap V(G[F])=\emptyset$. 
We remark that, given our definition of nooses as cycles of the radial graph, $G$ has only finitely many nooses, and thus the minimality in the last point above is w.r.t.\ a finite set of nooses. 
For readers familiar with the proof in \cite{gu2012improved}, we note that a similar definition appears in \cite{gu2012improved} except that the nooses are required only to separate $H$ from $G[F]$ instead of $C_1$, which is less restrictive.
We could have used exactly the same setup as in \cite{gu2012improved}---and in fact, this would improve slightly the constant term in our bound in \cref{thm:main induction}---but we found that separating $H$ from a contour at depth $1$ instead of the outer face helped simplify the exposition of our proof. 

The following lemma follows essentially from the proof of Lemma 4.4 in \cite{gu2012improved}. 

\begin{lem}[Consequence of proof of Lemma 4.4 in \cite{gu2012improved}]
    \label{lem:Separating_nooses}
    Let $G$ be a plane graph. Let $k, h$ be integers with $k\geq 3$ and $h\geq 2$. 
    Let $F$ be the outer face of $G$ and let $\mathcal{C}$ be the set of contours of $G$ at depth $h$ from $F$.  
    If $G$ does not contain the \cgrid{k}{h} as a minor, then there exists a collection $\mathcal{N}$ of nooses of $G$ with the following properties:
    \begin{enumerate}
        \item For every noose $N\in \mathcal{N}$, there exists $C\in\mathcal{C}$ such that $N$ is $(k, C)$-optimal; \label{prop:N}
        \item For every contour $C\in \mathcal{C}$, there exists $N\in \mathcal{N}$ such that $C$ is drawn inside $\Interior(N)$; \label{prop:C}
        \item For every two distinct nooses $N,N'\in \mathcal{N}$, we have $\interior(N)\cap \interior(N')=\emptyset$. \label{prop:N_Nprime}
    \end{enumerate}
\end{lem}
\begin{proof}
    First observe that if $\mathcal{C}=\emptyset$, then $\mathcal{N}=\emptyset$ satisfies all the required properties. Hence, suppose that $\mathcal{C}\neq \emptyset$.

    Let $C\in\mathcal{C}$ and let $C_1$ be the contour at depth $1$ from $F$ such that $C$ is contained in $\Interior(C_1)$.     
    As $G$ does not contain the \cgrid{k}{h} as a minor, and since $d_{\mathcal{R}_G}(V(C), V(C_1))\geq 2(h-1)$, 
    \cref{lem:cylinder_minor} ensures the existence of a noose $N$ with $|V_G(N)|\leq k-1$ that separates $C$ and $C_1$ in $G$.  
    Since $N$ satisfies property \ref{item:separates} of the definition of $(k, C)$-optimal, and since $|V_G(N)|\leq k-1$, we deduce that there exists a noose in $G$ that is $(k, C)$-optimal; 
    we let $N_C$ denote such a noose. 

    Let $\mathcal{N}$ be a maximal subset of nooses in $\{N_C \mid C\in\mathcal{C}\}$ bounding disks that are inclusion-wise maximal and pairwise distinct.     
    We claim that $\mathcal{N}$ satisfies the statement of the lemma. 
    Property~\ref{prop:N} holds by definition of $\mathcal{N}$. 
    To see that property~\ref{prop:C} holds, let $C\in \mathcal{C}$. 
    Then, seeing $N_C$ as a closed curve, we have that $C$ is drawn inside $\Interior(N_C)$. 
    If $N_C \in \mathcal{N}$, we are done. 
    Otherwise, by maximality of $\mathcal{N}$, there is $N\in \mathcal{N}$ such that $\Interior(N_C) \subset \Interior(N)$, and $N$ is then the desired noose.

    Finally, to show property~\ref{prop:N_Nprime}, let $N,N'\in \mathcal{N}$ be distinct. 
    Let $C, C'\in \mathcal{C}$ be such that $N$ is $(k, C)$-optimal and $N'$ is $(k, C')$-optimal.   
    We use Lemma 4.4 from \cite{gu2012improved} on $N$ and $N'$: 
    The proof of that lemma in \cite{gu2012improved} shows that 
    if $\interior(N)\cap \interior(N')\neq \emptyset$ 
    then either $N$ is not $(k, C)$-optimal or $N'$ is not $(k, C')$-optimal (using our terminology), 
    by producing a better noose using an uncrossing argument. 
    The newly created noose, seen as a cycle of the radial graph, is contained in $N \cup N'$; 
    in particular, it is still entirely drawn in the disk $\Interior(C_1)$, which is necessary for it to separate $C_1$ from $C$ or $C'$.          
\end{proof}

\subsection{Proof of Main Technical Theorem}
\label{ssec:proof main induction}

We proceed with the proof of \cref{thm:main induction}, which we restate here for convenience.  

\MainTechnicalTheorem* 

The proof is by induction on the size of $V(G)$.

Let $F$ be the outer face of $G$ and let $\mathcal{C}$ be the set of contours of $G$ at depth $h$ from $F$. Let $\mathcal{N}$ be the family of nooses obtained by applying \cref{lem:Separating_nooses}. Thus, every noose $N\in \mathcal{N}$ is $(k,C)$-optimal for at least one contour $C\in \mathcal{C}$, every contour $C\in \mathcal{C}$ is contained in the closed disk bounded by some noose $N\in \mathcal{N}$, and the disks bounded by nooses in $\mathcal{N}$ have pairwise disjoint interiors. 
(Note that possibly $\mathcal{C}=\emptyset$, in which case $\mathcal{N}=\emptyset$.) 
Let us point out that $|V_G(N)|\geq 2$ holds for every $N\in \mathcal{N}$, since $G$ is $2$-connected. 

For every noose $N\in\mathcal{N}$, choose a closed curve $\gamma_N$ in $\R^2$ that fulfills the following conditions:
\begin{enumerate}
    \item $\gamma_N$ intersects $G$ only in vertices of $G$; 
    \item $\Interior(\gamma_N) \subseteq \Interior(N)$;     
    \item $\gamma_N \cap N = V_G(N)$; that is, $\gamma_N$ intersects the curve $N$ precisely at the vertices in $V_G(N)$ and nowhere else, and      
    \item every vertex and every edge of $G$ contained in $\Interior(N)$ is also contained in $\Interior(\gamma_N)$.
\end{enumerate}
Note that such a curve $\gamma_N$ can easily be obtained from $N$ by \say{shrinking} slightly each section of $N$ between two consecutive vertices inside the disk bounded by $N$. 
Observe also that, for every two distinct nooses $N, N'\in \mathcal{N}$, the closed disks $\Interior(\gamma_{N}), \Interior(\gamma_{N'})$ only meet in vertices of $G$ and are otherwise disjoint (which was the reason for defining these closed curves). 

Our plan is, for every $N\in\mathcal{N}$,  to apply induction 
on the subgraph of $G$ captured by the closed disk $\Interior(\gamma_{N})$.  
However, we cannot simply take these subgraphs as such, because they are not guaranteed to be $2$-connected, and their outer faces might contain extra vertices not on $\gamma_{N}$ (which could cause these outer faces to have sizes bigger than $k$). 
To fix this, we will add extra edges in a natural way. 
This will possibly create some parallel edges, which is why we (temporarily) consider multigraphs in what follows. 

Next, for every noose $N\in\mathcal{N}$, we modify the graph $G$ as follows, resulting in a multigraph: 
Consider the vertices of $G$ on $\gamma_N$ in, say, clockwise order. 
For every two consecutive vertices $v, w$ in this cyclic order, 
\begin{itemize}
    \item if there is already an edge $vw$ drawn inside $\Interior(\gamma_N)$, remove it, 
    \item then add the edge $vw$ and draw it along the segment from $v$ to $w$ in $\gamma_N$ clockwise.
\end{itemize}
We let $C_N$ denote the cycle defined by the newly added edges, which is thus drawn along $\gamma_N$.  
(Note that $C_N$ has at least two edges, since $|V_G(N)|\geq 2$.) 
We let $\widetilde{G}$ denote the plane multigraph resulting from the above operation on every noose $N\in\mathcal{N}$. 
(The fact that this results in a plane drawing follows from the fact that the closed disks bounded by the curves $\gamma_{N}$ for $N\in \mathcal{N}$ only meet in vertices in $G$ and are otherwise disjoint.) 
Note that $\widetilde{G}$ may contain parallel edges, for three reasons: 
(1) the cycle $C_N$ consists of two parallel edges in case $|V_G(N)|=2$; 
(2) an edge $vw$ drawn outside $\Interior(\gamma_N)$ could exist in $G$ when a new edge $vw$ is added along a segment of $\gamma_N$, and   
(3) the same edge $vw$ can be added multiple times in the drawing if $v, w$ appear consecutively on multiple nooses in $\mathcal{N}$. 
Observe also that, on the other hand, for every noose $N\in\mathcal{N}$, the subgraph of $\widetilde{G}$ contained in $\Interior(\gamma_N)$ has parallel edges only if $|V_G(N)|=2$, and in this case the only parallel edges are the two edges in $C_N$; 
this will be important for applying induction later on since parallel edges are not allowed in the statement of \cref{thm:main induction}. 
Finally, observe that every tree-decomposition of $\widetilde{G}$ is also a tree-decomposition of $G$, and $F$ is also the 
outer face of $\widetilde{G}$. Thus, in what follows, we may focus on finding a tree-decomposition of $\widetilde{G}$ that has the desired properties.

Let $G'$ be the plane multigraph obtained from $\widetilde{G}$ by making every open disk $\interior(N)$ ($N\in \mathcal{N}$) a face $F_N$ bounded by the cycle $C_N$; that is, every vertex in $\interior(N)$ is removed, and every edge intersecting $\interior(N)$ is removed as well.  

Next, for every $N\in \mathcal{N}$, let $G_N$ be the plane multigraph that is the subgraph of $\widetilde{G}$ that is drawn inside $\Interior(N)$. 
As pointed out before, $G_N$ has no parallel edges, except if $|V_G(N)|=2$, and in this case the only parallel edges are two edges in $C_N$. Also, the outer face of $G_N$ is bounded by the cycle $C_N$, which has size 
 $|C_N|=|V_G(N)|\leq k-1$.
\begin{figure}[!htb]
\centering
\begin{subfigure}{\textwidth}
  \centering    
    \includegraphics[width=0.6\linewidth]{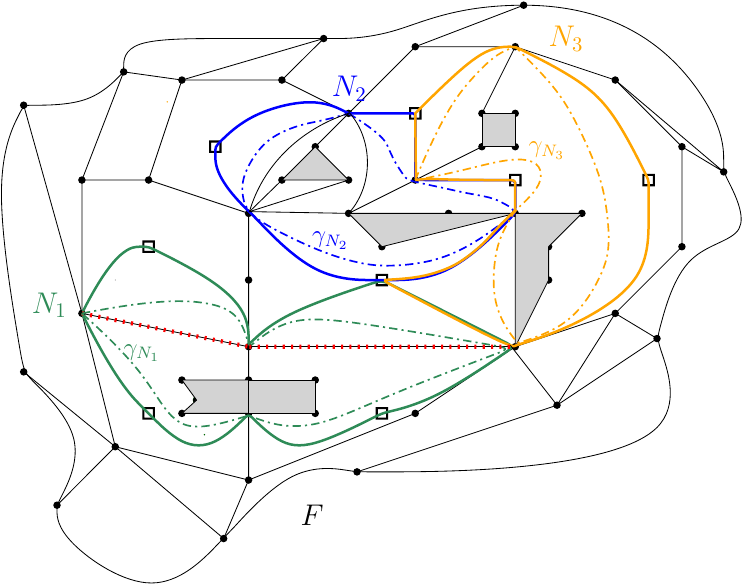}   
  \caption{
  For each $i\in \{1, 2, 3\}$, the curve $\gamma_{N_i}$ is the dash dotted closed curve that has the same color as the noose $N_i$.  The gray areas represent the interiors of the contours at depth $h$ that are inside the nooses $N_1, N_2$ and $N_3$.  Full circle vertices are vertices of $G$, while square vertices are vertices of $\mathcal R_G$ corresponding to faces of $G$.  The hatched edge is an edge that will be removed while creating $G_{N_1}$.   }
  \label{subfig:modification_G}
\end{subfigure}%

\vspace{2em}

\begin{subfigure}{.47\textwidth}
  \centering    
    \includegraphics[width=0.9\linewidth]{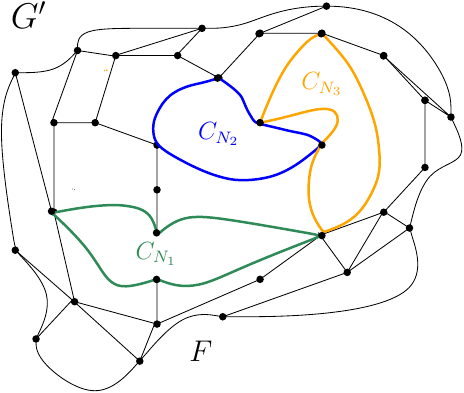}
  \caption{Illustration of the graph $G'$ obtained from the modifications of $G$ in \cref{subfig:modification_G}.  For each $i\in \{1, 2, 3\}$, the curve $\gamma_{N_i}$ has been replaced by a cycle $C_{N_i}$ of the same color, and every vertex an edge that lie inside $\interior(N_i)$ has been removed.
  } 
  \label{fig:subfig: G'}
\end{subfigure}
\hfill{}
\begin{subfigure}{.47\textwidth}
  \centering
  \includegraphics[width=0.9\linewidth]{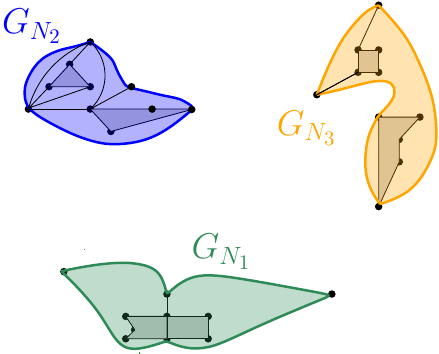}
 \caption{Illustration of the graphs $G_{N_1}, G_{N_2}$ and $G_{N_3}$.  For each $i\in \{1, 2, 3\}$, the boundary of the outer face of $G_{N_i}$ is a cycle drawn on $\gamma_{N_i}$.  Observe that the hatched edges in \Cref{subfig:modification_G} are indeed not present in $G_{N_1}$.}
  \label{subfig : G_{N_i}}
\end{subfigure}
\caption{Illustration of the modification of $G$ in 
    \cref{thm:main induction}.  }
\label{fig:examplesC1P}
\end{figure}

\begin{lem}\label{lem:vertices_of_G'_have_small_eccentricity}
    Let $v\in V(G')$.  In $\mathcal R_G$, the distance between $v$ and $F$ is at most $2h+1$.
\end{lem}
\begin{proof}
    Suppose for a contradiction that there exists $v\in V(G')$ such $v$ is at distance at least $2h+2$ from $F$ in $\mathcal R_G$.  Thus, $v\in V(G_h)$.  If $v$ is on the outer face of $G_h$, then, by \cref{lem:not_in_a_block}, $v$ is at distance exactly $2h+1$ from $F$ in $\mathcal R_G$, a contradiction.  Therefore, there exists a nontrivial block $B$ of $G_h$ such that $v\in V(G_h)$.  Let $C$ be the contour at depth $h$ which is the cycle delimiting the outer face of $B$.  Since $v$ is not on the outer face of $G_h$, $v$ lies in $\interior(C)$.  However, by definition of $\mathcal N$, there is a noose $N\in \mathcal N$ that separates $C$ from $F$ in $G$.  Therefore, $v\in \interior(N)$, and, by construction of $G'$, $v\not\in V(G')$, a contradiction. 
\end{proof}

\begin{lem}\label{lem: tree_dec_shallow_graph}
    There exists a tree-decomposition $\mathcal{T}'$ of $G'$ of width at most $\max\{3h+k+4, 2h+ 2k-1\}$ such that, 
    for every $N\in \mathcal{N}$, there exists a bag containing $V(F_N)=V_G(N)$. 
    Moreover, if the outer face $F$ is such that $|V(F)|\leq k-1$, then there is also a bag containing $V(F)$. 
\end{lem}
\begin{proof}
We first prove two useful results about shortests path in the radial graph of $G$. 
Afterwards we use them to construct a plane supergraph $H$ of $G'$ that has radius at most $h$, 
such that $F$ and every $F_N$ ($N\in \mathcal{N}$) remain faces of $H$. 
Then, the desired tree-decomposition will follow from \cref{cor:td_faces}.

\begin{claim}\label{clm:do_not_enter_the_noose}
    Let $v$ be a vertex of $V(G')$ and let $P$ be a shortest path in $\mathcal{R}_{G}$ between $v$ and $F$.  
    Suppose that $x, y$ are two consecutive vertices on $P$ and that, in $\mathcal{R}_{G}$, they are drawn inside the disk
    $\Interior(N)$ for some $N\in \mathcal{N}$. 
    Then $x,y$ are on the noose $N$ and they are consecutive on $N$. 
\end{claim} 

\begin{proofclaim}  
    By \cref{lem:vertices_of_G'_have_small_eccentricity}, the length of $P$ is at most $2h+1$.  Towards a contradiction, suppose $x, y$ do not satisfy the statement of the claim.  
    
    Let $P'$ be an inclusion-wise maximal subpath of $P$ such that $x, y\in V(P')$, the two endpoints of $P'$ belong to $N$, and every internal vertex of $P'$ is drawn in $\mathcal R_G$ in the open disk $\interior(N)$.  
    Note that $P'$ is well defined, since the two endpoints $F$ and $v$ of $P$ are drawn outside $\Interior(N)$ in $\mathcal R_G$. 
    Let $a, b$ be the endpoints of $P'$. 
    Thus $a, b$ are distinct and they are the only vertices of $P'$ on $N$.

    Let $P_1,P_2$ be the two $a$--$b$ paths contained in the cycle $N$ of $\mathcal{R}_G$ whose union is $N$. 
    For $i=1,2$, let $N_i:=P_i\cup P'$; observe that $N_i$ is a cycle in $\mathcal{R}_G$, or equivalently, a noose of $G$.  Since $P'$ is a subpath of $P$, it is a shortest path between $a$ and $b$ in $\mathcal R_G$.  In particular, $|V(P')|\leq |V(P_i)|$ for each $i\in \{1, 2\}$, and therefore the cycle $N_i$ of $\mathcal{R}_G$ has length at most that of $N$, for each $i\in \{1, 2\}$.  

    \begin{figure}[h!]
    \centering
    \includegraphics[width=0.65\linewidth]{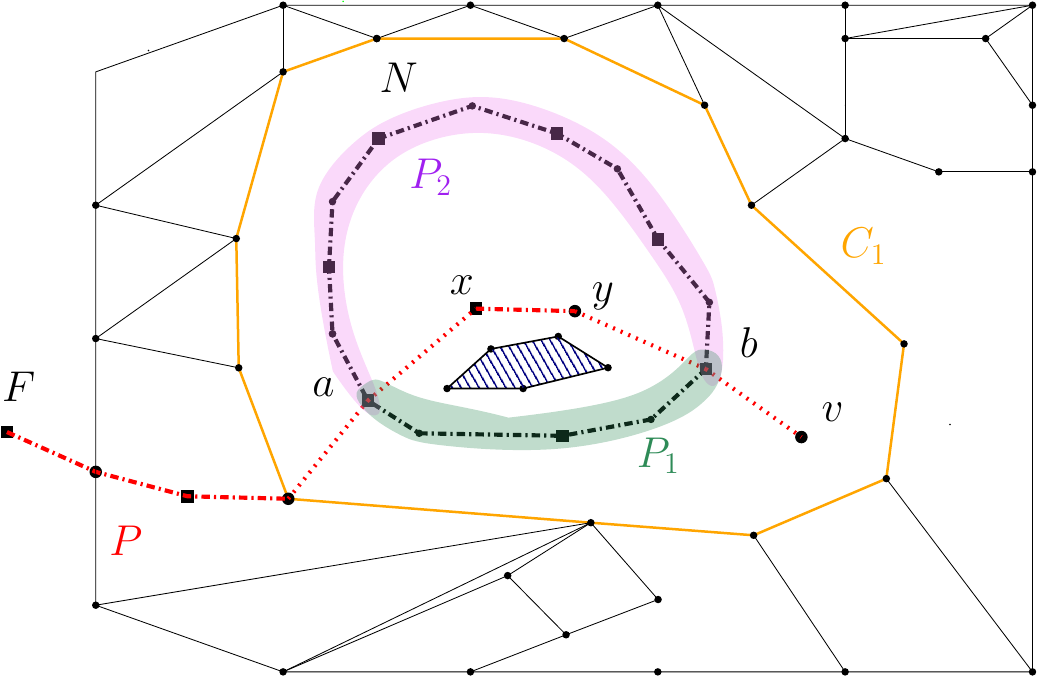}
    \caption{Schema of the hypothetical case inside the proof of \cref{clm:do_not_enter_the_noose}, when the path $P'$ is inside the noose $N$.  Full edges represent edges of the graph $G'$, dash dotted edges represent edges of $\mathcal R_G$, and dotted edges represent paths in $\mathcal R_G$ (potentially of length zero).  The block $B$ is hatched.}
    \label{fig:placeholder}
    \end{figure}

    Recall that $|V_G(N)|\leq k-1$ and that $N$ is $(k, C)$-optimal in $G$ for some $C\in \mathcal{C}$.  
    Thus, by \cref{lem:not_in_a_block}, every vertex on $C$ is at distance exactly $2h+1$ from $F$ in $\mathcal R_G$. 
    Hence, if $B$ is the block whose contour is $C$, then every vertex of $\mathcal{R}_G$ corresponding to an inner face of $B$ 
    is at distance at least $2h+2$ from $F$ in $\mathcal R_G$, and similarly every vertex of $\mathcal{R}_G$ corresponding to a vertex of $B-V(C)$ is at distance at least $2h+2$ from $F$ in $\mathcal R_G$. 
    In particular, none of these vertices of $\mathcal R_G$ is contained in the path $P$.  
    It follows that there exists $i\in \{1, 2\}$ such that $B$ is drawn inside the disk $\Interior(N_i)$ in $G$, and in particular $C$ is drawn inside $\Interior(N_i)$.  Without loss of generality, suppose $i=1$.  
    
    Let $C_1$ be the contour at depth $1$ in $G$ such that $C$ is drawn in $G$ inside the disk $\Interior(C_1)$. 
    Recall that the length of $N_1$ is at most that of $N$. 
    We will show that $N_1$ is a better noose than $N$, contradicting the fact that $N$ is $(k, C)$-optimal. 

    If $P'$ contains at least three vertices, then 
    $P'$ contains at least one vertex that is drawn in $\mathcal R_G$ in the open disk $\interior(N)$.  
    It follows that $P'$ is fully drawn inside $\interior(N)$, and hence that $N_1$ 
    is a noose separating $C$ from $C_1$ in $G$. 
    Moreover, $\interior(N_1)$ is strictly contained in $\interior(N)$, implying that $N$ is not $(k, C)$-optimal, a contradiction. 

    Now, assume that $P'$ consists only of the two vertices $a, b$.  
    Suppose, for contradiction, that $a$ and $b$ are not consecutive on $N$.  
    Then, $N_1$ has length strictly less than that of $N$, and in particular 
    $|V_G(N_1)|<|V_G(N)|$.  If the path $P'$, i.e.\ the edge $ab$ of $\mathcal R_G$, is fully drawn inside $\Interior(N)$, 
    then $N_1$ separates $C$ from $C_1$ in $G$, and we deduce that $N$ is not $(k, C)$-optimal, as in the previous paragraph. 
    We may thus assume that the edge $ab$ is drawn outside $\Interior(N)$ except for its two endpoints. 
    We claim that $N_1$ is nevertheless drawn inside $\Interior(C_1)$ in $\mathcal R_G$, implying that $N_1$ still separates $C$ from $C_1$ in $G$, and in turn that $N$ is not $(k, C)$-optimal, as desired. 
    This can be seen as follows. 
    First, note that the path $N_1-ab = P_1$ in $\mathcal R_G$ is drawn inside $\Interior(C_1)$. 
    Thus, if $N_1$ is not fully drawn inside $\Interior(C_1)$, it follows that the edge $ab$, as a curve, crosses the closed curve defined by $C_1$ at least twice. 
    However, since $C_1$ is a cycle of $G$ and $N_1$ is a cycle of $\mathcal R_G$, their only possible points of intersection, as closed curves in the plane, are vertices of $G$. 
    In particular, the edge $ab$ intersects the closed curve defined by $C_1$ at most once. 
    This completes the proof.   
\end{proofclaim}

\begin{claim}\label{cl:shortest_path_noose}
     Let $v$ be a vertex of $V(G')$ and let $P$ be a shortest path in $\mathcal{R}_{G}$ between $v$ and $F$.  
    Suppose that $x, y$ are two vertices of $V(G)\cap V(P)$ that are at distance $2$ on $P$, and that are drawn inside the disk 
    $\Interior(N)$ for some $N\in \mathcal{N}$. 
    Then $x,y$ are on $N$ and they are at distance $2$ on the cycle $N$.
\end{claim}
\begin{proofclaim}
    Let $z$ be the unique vertex of $P$ between $x$ and $y$.  First, assume that $z$ is drawn in $\Interior(N)$ in $\mathcal{R}_{G}$.  By \cref{clm:do_not_enter_the_noose} applied on $x,z$, and then on $z,y$, we deduce that $x, y$ and $z$ are vertices of $N$, that $x,z$ are consecutive on $N$, and that $z,y$ are consecutive on $N$.  In this case, the result directly follows.

    Now, assume that $z$ is not drawn in $\Interior(N)$.  Let $P'$ be the subpath of length $2$ of $P$ containing the vertices $x, z$ and $y$.  Let $P_1,P_2$ be the two $x$--$y$ paths contained in the cycle $N$ of $\mathcal{R}_G$ whose union is $N$. 
    For $i=1,2$, let $N_i:=P_i\cup P'$.  Since $z$ is not drawn in $\Interior (N)$, there exists $i\in \{1, 2\}$ such that $\Interior(N)\subset \Interior (N_i)$.  Without loss of generality, $i=1$.  Let $C\in \mathcal{C}$ be such that $N$ is $(k, C)$-optimal in $G$.  Since $\interior(N)\subset \interior(N_1)$, we have that $C$ is drawn entirely inside the disk $\Interior(N_1)$.  Let $C_1$ be the contour at depth $1$ from $F$ in $G$ such that $C$ is drawn inside $\Interior(C_1)$.  First, assume that $N_1$ is entirely drawn in the disk $\Interior(C_1)$.  Thus, $N_1$ separates $C$ and $C_1$ in $G$.  
    Since $N$ is $(k, C)$ optimal, it follows that $|V_G(N)|\leq |V_G(N_1)|$ and thus $|V(P_2)|\leq |V(P')|=3$.  Therefore, $|V(P_2)|=3$, thus $x$ and $y$ are at distance $2$ on the cycle $N$, and the result follows.

    Now, assume that $N_1$ is not entirely drawn inside $\Interior(C_1)$.  
    We will show that this leads to a contradiction. 
    Since $N$ is drawn inside $\Interior(C_1)$, this implies that $P'$, seen as a curve, is not entirely drawn inside $\Interior(C_1)$.  However, $C_1$ is a cycle of $G$ and $P'$ is a path of $\mathcal R_G$, and thus $P'$ can intersect $C_1$ only in points corresponding to vertices of $G$.  Since $P'$ is not entirely drawn inside $\Interior(C_1)$ but its endpoints $x, y$ are drawn in $\Interior(C_1)$, $P'$ and $C_1$ intersect at least twice as curves.  
    Therefore, they intersect exactly twice, and the two intersections are $x$ and $y$.  In particular, $x$ and $y$ belong to $C_1$, and thus, by \cref{lem:not_in_a_block}, they are at distance exactly $3$ from $F$ in $\mathcal R_G$.  This is a contradiction, since $x$ and $y$ are distinct vertices of a shortest path from $v$ to $F$ in $\mathcal R_G$, which implies that their distance to $F$ in $\mathcal R_G$ cannot be the same. 
  \end{proofclaim}

In the next claim, we build the supermultigraph $H$ of $G'$ mentioned before. 

\begin{claim}\label{cl:pseudo-triangulation}
There exists a plane multigraph $H$ such that 
\begin{enumerate}
        \item $V(G')=V(H)$; \label{prop:V}
        \item $E(G')\subseteq E(H)$; \label{prop:E}
        \item $F$ is a face of $H$; \label{prop:F}
        \item $F_N$ is a face of $H$ for every $N\in \mathcal{N}$, and \label{prop:F_N}
        \item every vertex $v\in V(H)$ is at distance at most $h$ from $V(H[F])$ in $H$. \label{prop:h}
    \end{enumerate}
\end{claim}

\begin{proofclaim}
    Let $T$ be a BFS tree of $\mathcal R_{G}$, rooted in $F$. Let us construct $H$ by adding edges to $G'$ as follows. 
    
    Consider every face $Z$ of $G$ distinct from $F$ such that $G[Z]$ contains at least two vertices of $G'$ and the parent of $Z$ in $T$, denoted by $z$, is in $V(G')$. 
    The face $Z$ corresponds to a vertex of $\mathcal{R}_G$, hence to a point of the plane; 
    let $Z'$ be the face of $G'$ containing this point. 
    As $z\in V(G')$, for each $N\in \mathcal{N}$ the only intersection of $\Interior (\gamma_N)$ and the edge $zZ\in E(\mathcal{R}_G)$ is $z$, by construction of $G'$. Hence, $zZ$ is a curve in the plane that is entirely contained in (the closure of) $Z'$, and we conclude that $z\in V(G'[Z'])$. 
    Now, if $Z'$ is distinct from all the faces $F_N$ with $N\in \mathcal{N}$, we add edges as follows (see \cref{fig:H} for an illustration): 
    For every vertex $x$ in $G'[Z']$ such that $x\neq z$ and $x$ is a child of $Z$ in $T$, add the edge $zx$, drawn inside the face $Z'$.  
    Note that all the newly added edges can be drawn in a planar way, since they are all incident to $z$. 

    \begin{figure}[h]
    \centering
    \includegraphics[width=0.3\linewidth]{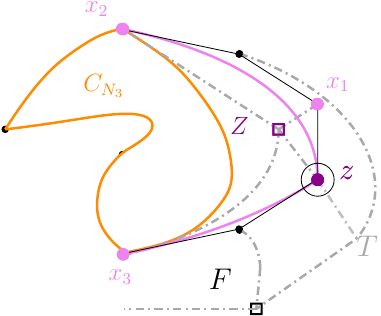}
    \caption{Additional edges drawn in the face $Z$ to create $H$.}
    \label{fig:H}
    \end{figure}
    
    Let $H$ be the multigraph resulting from the above edge additions to $G'$.     
    Let us show that $H$ fulfills the properties of the claim. 
    Properties \ref{prop:V}--\ref{prop:F_N} directly follow from its construction. 

    It remains to prove Property \ref{prop:h}. 
    Let $v\in V(G')$.  
    By \cref{lem:vertices_of_G'_have_small_eccentricity}, the distance between $v$ and $F$ in $\mathcal R_G$ is at most $2h+1$.  Let $P$ be the $v$--$F$ path in $T$, which is thus a shortest path in $\mathcal R_G$ from $v$ to $F$.  To show that Property \ref{prop:h} holds for $v$, it is enough to argue that every vertex in $V(P)\cap V(G)$ is also a vertex of $G'$, and that, for every two vertices $u, w\in V(P)\cap V(G')$ at distance $2$ on $P$, there exists an edge in $H$ having these two vertices as endpoints.  

    First, let us show that every vertex in $V(P)\cap V(G)$ is also in $V(G')$. Towards a contradiction, assume that there exists $u\in V(P)$, such that $u\in V(G)\setminus V(G')$. By construction of $G'$, the vertex $u$ is drawn inside $\interior(N)$ for some $N\in \mathcal N$. 
    Recalling that $F$ is drawn in $\mathcal R_G$ outside $\interior(N)$, we deduce that the $u$--$F$ subpath of $P$ has length at least $3$. 
    Let $Z$ denote the neighbor of $u$ on the latter subpath. 
    Then, $Z$ is drawn in $\mathcal R_G$ inside $\interior(N)$, and it follows from \cref{clm:do_not_enter_the_noose} that $u,Z$ are both on $N$, a contradiction.

    Now, let us show that, for every two vertices of $u,w\in V(G')$ at distance $2$ on $P$, the edge $uw$ exists in $H$. 
    Without loss of generality, assume that $u$ is closer to $F$ than $w$ in $P$, and let $Z$ be the face of $G$ corresponding to the vertex of $P$ between $u$ and $w$.
    Then $u$ is the parent of $Z$ in $T$ and $w$ is a child of $Z$ in $T$. 
    In particular, the face $Z$ was considered in the definition of $H$ above; let $Z'$ be the corresponding face in $G'$. 
    
    If $Z'=F_N$ for some $N\in \mathcal{N}$, then $u, w\in V_G(N)$. By \cref{cl:shortest_path_noose}, $u$ and $w$ are consecutive on $N$ and, by construction of $G'$, the edge $uw$ is in $G'$, and thus also in $H$. 
    If $Z'\neq F_N$ for every $N\in \mathcal{N}$, then the edge $uw$ was added inside the face $Z'$ when constructing $H$. 
    Therefore, $uw\in E(H)$ in both cases, as desired. 
    This concludes the proof. 
\end{proofclaim}

The graph $H$ is such that $F$ is at distance at most $h$ from any vertex, and the faces in $\mathcal F$ are faces of length at most $k-1$.  We can apply \Cref{cor:td_faces} to $H$ and $\mathcal{F}$, the set of faces of size at most $k-1$ to obtain a tree-decomposition $\mathcal{T}'$ of width at most $ \max\{3h+k+4, 2 h+ 2k-1\}$ such that for every face in $\mathcal{F}$, and $F$ if $|V(F)|\leq k-1$, there exists a bag containing its vertex set.  As $G'$ is a submultigraph of $H$ on the same vertex set, $\mathcal T'$ is also a tree-decomposition of~$G'$.
\end{proof}

The following lemma will allow us to apply induction on each graph $G_N$ ($N\in \mathcal{N}$). 
The last property in the lemma, about parallel edges of $G_N$, has already been justified when defining $G_N$; 
it is repeated in the lemma statement for emphasis. 

\begin{lem}\label{lem:tree_decomposition_G_N}    
    Let $N\in \mathcal{N}$. Then $G_N$ is a plane $2$-connected multigraph that does not contain the \cgrid{k}{h} as a minor and 
    satisfies
    \[
    |V(G_N)| < |V(G)| \quad \textrm{ and } \quad |V(C_N)|\leq k-1.   
    \]
    Moreover, if $G_N$ has parallel edges, then $|V(C_N)|=2$ and the only parallel edges are the two edges in $C_N$. 
\end{lem}
\begin{proof}
Let $C\in \mathcal{C}$ be such that $N$ is $(k,C)$-optimal. 
Let $v_1, v_2, \ldots , v_\ell$ be the vertices of $C_N$, in clockwise order around $C_N$. 
Thus $2\leq \ell \leq k-1$. 

Let $C_1$ be the contour at depth $1$ such that $C$ is drawn in the interior of $C_1$.  Since $N$ is $(k,C)$-optimal, every noose $N'$ of $G$ separating $C_1$ from $C$ in $G$ satisfies $|V_G(N')|\geq \ell$.   
By \cref{lem:planar_Menger}, there exists a collection of $\ell$ pairwise vertex-disjoint $V(C)$--$V(C_1)$ paths in $G$, and, since $N$ separates $C$ from $C_1$, each of these paths contains exactly one vertex of $V(C_N)$. 
Let $\mathcal{P}=\{P_1,\ldots,P_\ell\}$ denote such a collection. 
We may assume that $V(P_i)\cap V(C_N)=\{v_i\}$, for each $i\in [\ell]$. 
For each $i\in [\ell]$, let $w_i$ be the endpoint of $P_i$ in $V(C_1)$, and let $P'_i$ be the $v_i$--$w_i$ subpath of $P_i$. 
Then $\mathcal{P'}=\{P'_1,\ldots, P'_\ell\}$ is a collection of pairwise vertex-disjoint $V(C_N)$--$V(C_1)$ paths in $G$. 
By planarity, the vertices in $\{w_i\mid i \in [\ell]\}$ appear in clockwise order as $w_1, w_2, \dots, w_\ell$ along the cycle $C_1$. (To see this, consider the cylinder formed by the union of the two cycles $C_1$ and $C_N$ and the $\ell$ paths $P'_1,\ldots, P'_\ell$.) 
For each $i\in [\ell]$, let $Q_i$ be the path from $w_i$ to $w_{i+1}$ in $C_1$ in clockwise direction (where indices are taken cyclically).  

We use the paths defined above to show that $G_N$ is a minor of $G$:  
The only edges of $G_N$ that are possibly not in $G$ are edges of $C_N$. 
Now, let  
\[
J_N := (G_N - E(C_N)) 
\cup (P'_1 \cup \cdots \cup P'_\ell)
\cup (Q_1 \cup \cdots \cup Q_\ell). 
\]
In words, $J_N$ is the graph obtained from $G_N$ by removing all edges of $C_N$ and taking the union with all the paths $P'_1,\ldots, P'_\ell$ and $Q_1,\ldots, Q_\ell$. 
Then, $J_N$ is a subgraph of $G$ that contains $G_N$ as a minor. 
Hence, $G_N$ is a minor of $G$, as claimed. 

It follows that $G_N$ does not have a \cgrid{k}{h} minor, since $G$ does not. 

Let us show that $G_N$ is $2$-connected. 
This essentially follows from the fact that $G$ is $2$-connected and the existence of the cycle $C_N$ in $G_N$. 
First, we point out that $|V(G_N)|\geq 3$ since $N$ separates some contour $C\in \mathcal{C}$ from $F$ in $G$, and $|V(C)|\geq 3$. 
The fact that $G_N$ is connected follows easily from the fact that $G$ is connected. 
It remains to show that $G_N$ has no cut vertex. 
Suppose that $v$ is a cut vertex of $G_N$. 
In particular, $v$ separates in $G_N$ the cycle $C_N$ from some vertex $z$ of $G_N - (V(C_N) \cup \{v\})$. 
(Note that possibly $C_N$ consists of two parallel edges, which is not an issue for this argument.) 
Then, $v$ also separates $z$ from $V(C_N)$ in $G$, contradicting the fact that $G$ is $2$-connected. 
Thus, there is no such cut vertex, and therefore $G_N$ is $2$-connected. 

Finally, note that $|V(G_N)| \leq |V(G)| - |V(G[F])| < |V(G)|$, since $C_N$ is vertex-disjoint from $V(G[F])$.   
\end{proof}

Let $N\in\mathcal{N}$. 
Our goal now is to apply induction on $G_N$. 
If $|V_G(N)|\geq 3$, then \Cref{lem:tree_decomposition_G_N} ensures that $G_N$ satisfies the hypotheses of \Cref{thm:main induction},  and we can apply the induction on $G_N$ with its outer face $F_N$, resulting in a tree-decomposition $\mathcal T_N$ of $G_N$ of width at most $\max\{3h+k+4,2h+2k-1\}$ where $V_G(N)$ is contained in a bag.  
If $|V_G(N)|= 2$, then let $e,e'$ be the two parallel edges composing the cycle $C_N$. 
It follows then from  \Cref{lem:tree_decomposition_G_N} that $G_N-e$ satisfies the hypotheses of \Cref{thm:main induction}, and we can apply the induction on $G_N-e$ resulting in a tree-decomposition $\mathcal T_N$ of $G_N-e$ of width at most $\max\{3h+k+4,2h+2k-1\}$, with no specific guarantee on the vertices of the outer face of $G_N-e$ in this case. 
Note however that, by definition of a tree-decomposition, there is a bag of $\mathcal T_N$ containing both endpoints of $e'$; that is, this bag contains $V_G(N)$. 
Therefore, in both cases, $V_G(N)$ is contained in a bag of the tree-decomposition $\mathcal T_N$. 

To complete the proof of \cref{thm:main induction}, we show that there exists a tree-decomposition of $\Tilde{G}$ respecting the conditions of \cref{thm:main induction}.  Since $\Tilde G$ contains $G$ as a subgraph, the result will directly follow.  
The tree-decomposition of $\Tilde{G}$ is obtained by combining the tree-decomposition $\mathcal{T}'$ of $G'$ given by \cref{lem: tree_dec_shallow_graph} and  
the tree-decompositions $\mathcal{T}_N$ of $G_N$ (with possibly an edge removed if $|V_G(N)|= 2$) for each $N\in \mathcal N$ in the natural 
way: For each $N\in \mathcal N$, consider a bag of $\mathcal{T}_N$ containing $V_G(N)$ and a bag of $\mathcal{T}'$ containing $V_G(N)$, and add an edge between the corresponding nodes of their respective trees. 
It is easily verified that this results in a tree-decomposition of $\Tilde{G}$ with the desired properties. 
This concludes the proof of \cref{thm:main induction}.

\bibliographystyle{abbrvnat}
\bibliography{bibliography}

\end{document}